\documentclass{amsart}

\usepackage{amsmath}
\usepackage{amsfonts}
\usepackage{amssymb}
\usepackage{amsthm}
\usepackage{comment}
\usepackage{epsfig}
\usepackage{psfrag}
\usepackage{mathrsfs}
\usepackage{amscd}
\usepackage[all]{xy}
\usepackage{rotating}
\usepackage{lscape}
\usepackage{amsbsy}
\usepackage{verbatim}
\usepackage{moreverb}
\usepackage{color}
\usepackage{bbm}
\usepackage{eucal}

\usepackage{tikz}
\usetikzlibrary{patterns,shapes.geometric,arrows,decorations.markings}
\usepackage{tikz-3dplot}

\usepackage{caption}
\usepackage{subcaption}

\colorlet{lightgray}{black!15}

\tikzset{->-/.style={decoration={
  markings,
  mark=at position .5 with {\arrow{>}}},postaction={decorate}}}
\tikzset{midarrow/.style={decoration={
    markings,
    mark=at position {#1} with {\arrow{>}}},postaction={decorate}}}

\newtheorem{theorem}{Theorem}[section]
\newtheorem{prop}[theorem]{Proposition}
\newtheorem{lemma}[theorem]{Lemma}
\newtheorem{cor}[theorem]{Corollary}
\newtheorem{conj}[theorem]{Conjecture}

\theoremstyle{definition}
\newtheorem{definition}[theorem]{Definition}

\newtheorem{observation}[theorem]{Observation}

\newtheorem{terminology}[theorem]{Terminology}
\newtheorem{remark}[theorem]{Remark}
\newtheorem{example}[theorem]{Example}

\theoremstyle{remark}

\definecolor{orange}{rgb}{.95,0.5,0}
\definecolor{light-gray}{gray}{0.75}
\definecolor{brown}{cmyk}{0, 0.8, 1, 0.6}
\definecolor{plum}{rgb}{.5,0,1}

\DeclareMathOperator{\Link}{\sf Link}
\DeclareMathOperator{\Fin}{\sf Fin}

\DeclareMathOperator{\Vect}{\cV{\sf ect}}

\DeclareMathOperator{\pr}{\mathsf{pr}}
\DeclareMathOperator{\ev}{\mathsf{ev}}

\DeclareMathOperator{\Alg}{\sf Alg}

\DeclareMathOperator{\Psh}{\sf PShv}

\DeclareMathOperator{\Aut}{\sf Aut}
\DeclareMathOperator{\colim}{{\sf colim}}

\DeclareMathOperator{\limit}{{\sf lim}}

\DeclareMathOperator{\Fun}{{\sf Fun}}

\DeclareMathOperator{\Map}{{\sf Map}}

\DeclareMathOperator{\exit}{\sf Exit}
\DeclareMathOperator{\Exit}{\bcE{\sf xit}}

\DeclareMathOperator{\obj}{{\sf obj}}

\DeclareMathOperator{\fCat}{{\sf fCat}}
\DeclareMathOperator{\Cat}{{\sf Cat}}

\DeclareMathOperator{\kEnd}{\it k{\sf End}}

\DeclareMathOperator{\uno}{\mathbbm{1}}

\DeclareMathOperator{\Ar}{{\sf Ar}}

\DeclareMathOperator{\Shv}{\sf Shv}
\DeclareMathOperator{\cShv}{\sf cShv}

\DeclareMathOperator{\op}{\mathsf{op}}

\DeclareMathOperator{\cBun}{{\sf c}\cB\mathsf{un}}
\DeclareMathOperator{\Bun}{\cB\mathsf{un}}

\DeclareMathOperator{\oC}{\ov\sC}

\DeclareMathOperator{\cls}{\mathsf{cls}}

\DeclareMathOperator{\opn}{\mathsf{open}}
\DeclareMathOperator{\emb}{\mathsf{emb}}

\DeclareMathOperator{\cbl}{\mathsf{cbl}}

\DeclareMathOperator{\pcbl}{\mathsf{p.cbl}}
\DeclareMathOperator{\Tang}{\cT{\sf ang}}
\DeclareMathOperator{\fTang}{\fT{\sf ang}}

\DeclareMathOperator{\Mfd}{{\cM}\mathsf{fd}}
\DeclareMathOperator{\cMfd}{{\sf c}{\cM}\mathsf{fd}}
\DeclareMathOperator{\Mfld}{{\cM}\mathsf{fld}}

\DeclareMathOperator{\Emb}{\mathsf{Emb}}

\DeclareMathOperator{\strat}{\mathsf{Strat}}
\DeclareMathOperator{\Strat}{\cS\mathsf{trat}}

\DeclareMathOperator{\spaces}{\cS\mathsf{paces}}
\DeclareMathOperator{\Spaces}{\cS\mathsf{paces}}

\DeclareMathOperator{\Spectra}{\cS\mathsf{pectra}}

\DeclareMathOperator{\Disk}{\cD{\mathsf{isk}}}

\DeclareMathOperator{\cDisk}{{\sf c}\cD{\mathsf{isk}}}

\DeclareMathOperator{\vfr}{\sf vfr}
\DeclareMathOperator{\fr}{\sf fr}
\DeclareMathOperator{\sfr}{\sf sfr}

\DeclareMathOperator{\Bord}{\cB{\sf ord}}

\def\ot{\otimes}

\DeclareMathOperator{\oo}{\infty}

\DeclareMathOperator{\tr}{\triangleright}
\DeclareMathOperator{\tl}{\triangleleft}

\newcommand{\lag}{\langle}
\newcommand{\rag}{\rangle}

\newcommand{\w}{\widetilde}

\newcommand{\ov}{\overline}

\newcommand{\ra}{\rightarrow}

\newcommand{\xra}{\xrightarrow}
\newcommand{\xla}{\xleftarrow}

\def\cB{\mathcal B}\def\cC{\mathcal C}\def\cD{\mathcal D}
\def\cE{\mathcal E}\def\cF{\mathcal F}
\def\cK{\mathcal K}
\def\cM{\mathcal M}
\def\cS{\mathcal S}\def\cT{\mathcal T}
\def\cV{\mathcal V}

\def\DD{\mathbb D}

\def\NN{\mathbb N}
\def\RR{\mathbb R}\def\SS{\mathbb S}

\def\ZZ{\mathbb Z}

\def\sB{\mathsf B}\def\sC{\mathsf C}
\def\sH{\mathsf H}

\def\sO{\mathsf O}
\def\sT{\mathsf T}

\def\bdelta{\mathbf\Delta}
\def\bDelta{\mathbf\Delta}

\def\fB{\frak B}\def\fC{\frak C}

\def\fT{\frak T}
\def\fX{\frak X}
\def\fZ{\frak Z}

\def\bcE{\boldsymbol{\mathcal E}}

\DeclareMathOperator{\btheta}{\boldsymbol{\Theta}}
\DeclareMathOperator{\bTheta}{\boldsymbol{\Theta}}
\DeclareMathOperator{\adj}{{\sf adj}}
\DeclareMathOperator{\smsh}{\wedge}

\begin{document}

\title{The cobordism hypothesis}

\author{David Ayala \& John Francis}

\address{Department of Mathematics\\Montana State University\\Bozeman, MT 59717}
\email{david.ayala@montana.edu}
\address{Department of Mathematics\\Northwestern University\\Evanston, IL 60208}
\email{jnkf@northwestern.edu}

\thanks{This work was supported by the National Science Foundation under award 1508040. Parts of this work were done when the authors were visitors at the Universit\'e Pierre et Marie Curie, the Hausdorff Institute, MSRI, and IMPA}

\begin{abstract}
Assuming a conjecture about factorization homology with adjoints, we prove the cobordism hypothesis, after Baez--Dolan, Costello, Hopkins--Lurie, and Lurie.

\end{abstract}



\maketitle

\tableofcontents

\section{Introduction}

In this paper, we show how the theory of factorization homology with adjoints implies the cobordism hypothesis. The cobordism hypothesis for topological quantum field theories is an analogue of the Eilenberg--Steenrod axioms for homology: the Eilenberg--Steenrod axioms state that a homology theory is uniquely determined by its value on a point; the cobordism hypothesis states that a topological quantum field theory is uniquely determined by its value on a point. The following table presents analogous elements:

\medskip

\begin{center}
    \begin{tabular}{|p{6cm}|  p{7.5cm} | }
    \hline
        {\bf Homology} & {\bf TQFT}  \\ \hline
    $\Spaces$ & $\Bord_n^{\fr}$  \\ \hline
    Eilenberg--Steenrod axioms: \newline $\ev_\ast \colon \Fun^{\sf exc}(\Spaces,\cV)\simeq \cV~.$ & Cobordism hypothesis: \newline $\ev_\ast\colon \Fun^{\ot}(\Bord_n^{\fr}, \fX) \simeq \obj(\fX)~.$    \\ \hline
        Cellular homology; 
        \newline cell complexes; 
        \newline cellular approximation.
        & 
        Composition of unit/counits; 
        \newline handlebodies; 
        \newline parametrized Morse theory. 
        \\ 
        \hline
       Eilenberg--MacLane homology: 
       \newline $\sH_{\ast}(M;A) \simeq  \underset{k\to \oo} \limit \pi_{\ast+k} (M_+\wedge \sB^k A)~;$ \newline $\SS\simeq \underset{k\to \oo} \limit  \Omega^k S^k~.$  & Factorization homology: \newline 
       $\Fun^{\ot}(\Bord_n^{\fr},\fX) \simeq \underset{k\to \infty}\limit\Fun_\ast(\Tang_{n\subset n+k}^{\fr}, \fB^k\fX)~;$ \newline
       $\Bord_n^{\fr} \simeq \underset{k\to \oo} \limit \Omega^k \fC(\RR^k)~.$  \\ \hline
        Euler characteristic: 
        \newline 
        $\chi(M)\colon \Bbbk  \xra{{\sf e}(M)} \sH^\ast(M;\Bbbk) \xra{\int_M} \Bbbk$
        \newline
        ($M$ closed, $\Bbbk$ a field)
        & 
        Partition function:
        \newline
        $Z_\fX(M)\colon \Bbbk \xra{\rm unit} \displaystyle \int_M \fX  \xra{\rm trace} n{\sf End}_\fX(\Bbbk)$
        \newline
        ($\fX$ a $\Bbbk$-linear $(\oo,n)$-category) \\ \hline
    \end{tabular}
\end{center}

\medskip

Since the 1980s, approaches to constructing topological quantum field theories from categorical data have relied on Smale's theory of handlebodies \cite{smale}, a smooth manifold analogue of cell complex structures in homotopy theory.
The essential goal is the following: 
given a compact $n$-manifold $M$ and a suitable object $x\in \fX$ in a higher category, construct the partition function $Z_x(M)$, which is an $n$-morphism in $\fX$.
The essential approach is this:
\begin{itemize}
\item[~]
Decompose $M$ as a sequence of handle attachments. 
To each handle of index-$k$, associate a $k$-morphism in $\fX$ which is the unit/counit of an adjunction determined from dualizability of the object $x\in \fX$. 
Composing these unit and counits gives $Z_x(M)$. 
\end{itemize}
This is how Reshetikhin \& Turaev constructed invariants of 3-manifolds from representations of quantum groups in \cite{rt2} and \cite{rt1}, and how Baez \& Dolan originally conceived of the cobordism hypothesis. 
See also \cite{craneyetter}, \cite{cranefrenkel}.
This approach is analogous to cellular homology.
There, the goal is this: given a space $M$ and an abelian group $A$, construct the graded abelian group $\sH_\ast(M;A)$, which is the homology of $M$.
The approach of cellular homology is this:
\begin{itemize}
\item[~]
Decompose $M$ as a sequence of cell attachments.
To each $k$-cell, associate the graded abelian group $A[k]$.  
These associations assemble as a chain complex; differential is given by the degrees of attaching maps; its homology is $\sH_\ast(M;A)$.
\end{itemize}
Both of these approaches are well-suited to explicit computation once a manifold/space has been decomposed into handles/cells.
However, it is difficult to show that they are well-defined invariants of $M$, independent of a choice of decomposition. 
To show cellular homology is independent of the choice of cell decomposition requires Whitehead's cellular approximation theorem.\footnote{
         Even showing the Euler characteristic of a manifold is a well-defined invariant is, essentially, as difficult as showing it is computed as the integral of the Euler class.
         The analogy between the Euler characteristic and the partition function is the composition involving factorization homology:        
        \[
        \Bbbk \simeq \displaystyle \int_\emptyset \fX  \xra{\emptyset \hookrightarrow M} \displaystyle \int_M \fX  \simeq \int_{\RR^k\times M} \fB^k \fX  \xra{\RR^k\times M \hookrightarrow \RR^{n+k}} \displaystyle \int_{\RR^{n+k}}\fB^k \fX \simeq \int_{\RR^n}\fX \simeq n{\sf End}_\fX(\Bbbk)~.
        \]
        }
The problem of well-definedness in TQFT is the essential difficulty of Baez--Dolan's cobordism hypothesis. 
This problem becomes harder in an $\oo$-categorical formulation of the cobordism hypothesis, e.g., working at the chain level and insisting that diffeomorphism groups of manifolds act suitably continuously. 
The proposed solution to this $\infty$-categorical phrasing, outlined by Lurie~\cite{lurie.cobordism}, is of this nature.
Accordingly, in this approach to a classification to TQFTs, one must address intricate problems concerning index-filtered parametrized Morse theory:
show that the space of framed Morse functions on a compact manifold is contractible (see~\cite{eliash}); show that this space can be filtered according to index of critical points with tractable filtration layers.  

\medskip

There are other approaches to proving Eilenberg--Steenrod's theorem, aside from cellular homology and cellular approximation. 
One other such approach is to give a definition of homology that is manifestly homotopy-invariant: one such definition is Eilenberg--MacLane homology.
In analogy, one could seek some homology-type construction of the field theory associated to a suitable object $x\in \fX$ in a higher category, that is manifestly well-defined.  
This is an essential goal of our program.

\medskip

The following principle guides this program: 
\begin{itemize}
\item[~]
Field theories can be constructed from the homology of sheaves on moduli spaces of stratifications of manifolds. 
\end{itemize}
In fact, Eilenberg--MacLane homology demonstrates this principle.  
The Eilenberg--MacLane spaces $\sB^k A$ can be organized as an infinite-loop space presentation of the abelian group $A\simeq \Omega^k\sB^k A$.
One can then construct a locally constant cosheaf of infinite-loop spaces the moduli space of points in a space $M$, whose stalk at each point is $A$.
Integrating this cosheaf over this moduli space, implemented by cosheaf homology, results in an infinite-loop space, the homotopy groups of which is the homology $\sH_\ast(M;A)$.

\medskip

Our principle is inspired by the work of Beilinson \& Drinfeld in conformal field theory in~\cite{bd}.  
There, for each algebraic curve $M$, they consider its Ran space---a moduli space of finite subsets in $M$.  
Sheaves on this Ran space can be constructed from chiral or vertex algebras in algebraic geometry, and in the case that $M$ is a framed $n$-manifold from $\cE_n$-algebras in $n$-manifold topology.
Their theory of chiral or factorization homology is then the sheaf homology over this Ran space. 
A special, somewhat degenerate, form of the cobordism hypothesis is provable using this form of factorization homology theory~\cite{oldfact}.  

\medskip

In the case of a smooth 4-manifold $M$, moduli of smoothly embedded surfaces in $M$, not merely finite subsets, is required to access the full geometry of $M$.
Both embedded surfaces and finite subsets of $M$ are simple examples of stratifications of $M$.
Thus, to effectuate our principle we address three questions. 
First, what is the moduli space of stratifications? 
Second, what are sheaves on this moduli space of stratifications?
Third, why are sheaves on this moduli space natural for TQFT?

\medskip

Different approaches are possible. For one, the moduli space of stratifications might be a topological space $\cM_M$, points of which are stratifications of $M$, and $\cM_M$ might then be equipped with a Gromov--Hausdorff topology. Toward the second question, one would then be seek a method for producing sheaves on this moduli space $\cM_M$. 
To construct continuous diffeomorphism invariants of $M$, these sheaves should not be arbitrary.
Namely, $\cM_M$ carries a natural stratification: two stratifications $S_0$ and $S_1$ belong to the same strata if they are smoothly isotopic;
we should then consider only those sheaves on $\cM_M$ which are constructible with respect to this stratification. 
Using the enter-path $\oo$-category associated to the stratification of $\cM_M$, such a $\cV$-valued constructible sheaf $\cC$ is equivalent to the data of a functor
\[
\exit(\cM_M) \overset{\cC}\longrightarrow \cV~.
\]
The resulting sheaf homology, which is the sought invariant of $M$, would then be equivalent to the colimit of this functor:
\[
\sH_\ast(\cM_M;\cC) ~ \simeq ~ \colim \bigl(\exit(\cM_M)\ra{\cC}\cV\bigr)~.
\]
One difficulty is that the topology of this moduli space $\cM_M$ is infinite-dimensional and non-Hausdorff: one must then confront the strangeness of sheaf theory on such topological spaces.

\medskip

Instead of building the topological space $\cM_M$ and then taking its enter-path $\oo$-category, we build the sought $\oo$-category $\exit(\cM_M)^{\op}$ directly, without reference to $\cM_M$.
To do so, we begin with the following question: 
while a point in $\cM_M$ should be a stratification of $M$, 
what should be an enter-path in $\cM_M$? 
That is, what is a stratified map $[0,1]\ra \cM_M$? 
(Here $[0,1]$ carries the asymmetric stratification with strata $\{0\}$ and $(0,1]$.) 
In as much as $\cM_M$ is a moduli space, this is the question of what such a map should classify.
Our answer to this formulation is: a {\it constructible bundle} over $[0,1]$. 
That is, we define an enter-path in the moduli space $\cM_M$ from a stratification $S_0$ of $M$ to a stratification $S_1$ of $M$ to be a stratification $S$ of $M\times [0,1]$ together with identifications
\[
\xymatrix{
(S_0\subset M\times \{0\})   \ar[rr]\ar[d]
&&
(S\subset M\times [0,1])     \ar[d]^-{\rm projection}
&&
(S_1  \subset M\times \{1\})  \ar[d]   \ar[ll]
\\
\{0\}   \ar[rr]
&&
[0,1]
&&
\{1\}   \ar[ll]
}
\]
and such that the restriction $S_{|(0,1]}\ra (0,1]$ splits as a product: $S_{|(0,1]} \cong S_1 \times (0,1]$. 
In the context of conically smooth stratifications, as developed in \cite{aft1}, functorial resolutions of singularities can be used to prove that 
that putative enter-paths can be composed: given two constructible bundles $(S_{01}\subset M\times[0,1]) \ra [0,1]$ and $(S_{12}\subset M\times [1,2])\ra [1,2]$ whose fibers over $\{1\}$ are identified, there exists a constructible bundle $(S\subset M)\ra \Delta^2$ over the standardly stratified 2-simplex (Definition~\ref{def.simplices}), together with identifications
\[
\xymatrix{
(S_{01}\subset M\times [0,1])  \ar[r]\ar[d]&
(S\subset M\times \Delta^2) \ar[d]&(S_{12}\subset M\times [1,2])\ar[d]\ar[l]\\
[0,1]\ar[r]&\Delta^2&\ar[l]
[1,2]  .
}
\]
The restriction $S_{02}:= S_{|[0,2]}\ra [0,2]$ then represents the composition of $S_{01}$ and $S_{12}$.
One must ensure that constructible bundles are closed under base-change and thus that $S_{02}\ra [0,2]$ is again a constructible bundle. 
The existence of base-change, and of this concatenation of constructible bundles, 
comprise the main technical results of~\cite{striat}.

\medskip

This leads to the second question, of how to construct coefficient systems on the moduli space of stratifications. Phrased in terms of enter-paths, such a coefficient system $\cC$ should take a value $\cC(S_0)$, the stalk, at each stratification $S_0$ of $M$; and for each constructible bundle $S\ra \Delta^1$ with fibers $S_0$ and $S_1$, $\cC$ should have a specialization map $\cC(S_0) \ra \cC(S_1)$. This is not a familiar algebraic structure, but one can make it so by introducing the appropriate notion of a framing of a stratified space, a \emph{vari-framing}, in which each stratum carries a framing.
There is an $\oo$-category $\Mfd_n^{\vfr}$ whose objects are vari-framed stratified spaces of dimension at most $n$, and whose morphisms are constructible bundles over $\Delta^1$ with a fiberwise vari-framing. 
One should think of the framed version of the moduli space $\cM_M$ as being captured by the $\oo$-overcategory $\Mfd_{n/M}^{\vfr}$.
There is an $\infty$-subcategory $\cMfd_n^{\vfr} \subset \Mfd_n^{\vfr}$ classifying the proper constructible bundles. 
In this case, $(\oo,n)$-categories define 
sheaves on the moduli space of vari-framed stratifications of each compact framed $n$-manifold $M$.

\begin{theorem}[\cite{emb1a}]
There is a fully faithful functor, factorization homology,
\[
\displaystyle \int \colon 
\Cat_n
\longrightarrow
\Fun(\cMfd_n^{\vfr},\Spaces)
\]
from $(\oo,n)$-categories into space-valued functors on compact vari-framed stratified $n$-manifolds.
\end{theorem}

To prove the cobordism hypothesis, it is essential to do two things: one, to mix the morphisms in the moduli space of stratifications with open embeddings in the manifold variable $M$; two, to relax the notion of the framing to accommodate adjoints. We will assume the following; see \S\ref{Preliminaries} for definitions for the terms involved. We have conservatively called the following a conjecture; however, we expect a proof of it to be shortly completed and appear in a forthcoming work.

\begin{conj}\label{conj.one}
Factorization homology with adjoints defines a fully faithful functor
\[
\int: \Cat_n^{\sf adj, \ast/}\hookrightarrow\Fun(\Mfd_n^{\sfr}, \Spaces)
\]
from pointed $(\oo,n)$-categories with adjoints to space-valued functors on solidly $n$-framed stratified manifolds.
This lies in a commutative diagram
\[
\xymatrix{
\Cat_n^{\sf adj, \ast/}\ar[rr]^-{\int}\ar[d]&&\Fun(\Mfd_n^{\sfr}, \Spaces)\ar[d]\\
\Cat_n^{\ast/}\ar[rr]^-\int &&\Fun(\Mfd_n^{\vfr},\Spaces)}
\]
with respect to factorization homology.
\end{conj}

We now address the third, and final, question of how sheaves on this moduli space of stratifications is relevant to TQFT. Namely, Conjecture~\ref{conj.one} implies the cobordism hypothesis, as put forth in \cite{baezdolan} and given precise form in \cite{lurie.cobordism}.

\begin{theorem}\label{bord-hyp}
Assuming Conjecture \ref{conj.one}, for each symmetric monoidal $(\infty,n)$-category $\fX$ with duals and adjoints, evaluation at the object $\ast \in \Bord_n^{\fr}$ defines an equivalence
\[
\ev_\ast\colon \Fun^{\ot}\bigl(\Bord_n^{\fr}, \fX\bigr) \xra{~ \simeq ~}  \obj(\fX)
\]
between the $(\infty,n)$-category of $\fX$-valued symmetric monoidal functors and the space of objects of $\fX$.
\end{theorem}

The cobordism hypothesis is a limiting case of the tangle hypothesis, and this is how we prove it. We prove the following form of the tangle hypothesis.

\begin{theorem}\label{tangle.hyp}
Assuming Conjecture \ref{conj.one}, for each pointed $(\infty,n+k)$-category $\uno\in\cC$ with adjoints, evaluation at the 
$k$-endomorphism
$(\{0\}\subset\RR^k) \in \Tang_{n\subset n+k}^{\fr}$ defines an equivalence
\[
\ev_{\RR^k}\colon \Fun_\ast\bigl(\Tang_{n\subset n+k}^{\fr}, \cC\bigr) 
\xra{~\simeq~}
\kEnd_{\cC}(\uno)
\]
between the $(\infty,n+k)$-category of $\cC$-valued pointed functors and the space of $k$-endomorphisms in $\cC$ of its distinguished object $\uno\in \cC$.

\end{theorem}

{\Small
\subsection{Notation}

The following are our primary notational objects.
\begin{itemize}
\item $\Cat$ is the $\oo$-category of $\oo$-categories (i.e., of $(\oo,1)$-categories).
\item $\Cat_n$ is the $\oo$-category of $(\oo,n)$-categories.
\item $c_k \in \Cat_n$ is the $k$-cell, which corepresents $k$-morphisms. 
\item $\Cat_n^{\adj}$ is the $\oo$-category of $(\oo,n)$-categories with adjoints for $k$-morphisms, $0<k<n$.
\item $\fCat_n$ is the $\oo$-category of flagged $(\oo,n)$-categories.
\item $\kEnd_{\cC}(c)$ is the space of $k$-endomorphisms in $\cC$ of an object $c\in\cC$.
\item $\fB^k \colon \Alg_{\cE_k}(\Cat_n)\ra \Cat_{n+k}^{\ast/}$ is the $n$-fold deloop of an $\cE_k$-monoidal $(\oo,n)$-category.
\item $\Omega^k\colon  \Cat_{n+k}^{\ast/}\ra \Alg_{\cE_k}(\Cat_n)$ is the $n$-fold loops of a pointed $(\oo,n+k)$-category.
\item $\Bord_n^{\fr}$ is the symmetric monoidal flagged $(\oo,n)$-category of $n$-framed cobordisms.
\item $\Tang_{n\subset n+k}^{\fr}$ is the pointed flagged $(\oo,n+k)$-category of codimension-$k$ framed tangles.
\item $\fTang_{n\subset n+k}^{\fr}$ is the $\oo$-category of codimension-$k$ framed tangles.
\item $[p] =\{0<1<\ldots<p\}$ is the totally ordered set with $p+1$ elements.
\item $\bdelta$ is the simplex category.
\item $\btheta_n= \bdelta^{\wr n}$ is Joyal's category, which generates $\Cat_n$ by \cite{rezk-n}.
\item $\Delta^p$ is the standardly stratified $p$-simplex.
\item $\Strat$ is the $\oo$-category of stratified spaces and spaces of stratified maps among them.
\item $\exit(X)$ is the exit-path $\oo$-category of a stratified space $X$; its $p$-simplices are stratified maps from $\Delta^p$ to $X$.
\item $\Bun$ is the $\oo$-category whose $p$-simplices are constructible bundles over $\Delta^p$.
\item $\Mfd_n^{\vfr}$ is the $\oo$-category whose $p$-simplices are constructible bundles over $\Delta^p$ equipped with a fiberwise vari-framing and fiber dimension bounded by $n$.
\item $\cMfd_n^{\vfr}$ is the $\oo$-category whose $p$-simplices are proper constructible bundles over $\Delta^p$ equipped with a fiberwise vari-framing and fiber dimension bounded by $n$.
\item $\Mfd_n^{\sfr}$ is the $\oo$-category whose $p$-simplices are constructible bundles over $\Delta^p$ equipped with a fiberwise solid $n$-framing.
\item $\cMfd_n^{\sfr}$ is the $\oo$-category whose $p$-simplices are proper constructible bundles over $\Delta^p$ equipped with a fiberwise solid $n$-framing.
\item $\displaystyle\int_M \cC$ is the factorization homology over $M$ with coefficients in $\cC$.
\end{itemize}

}

\subsection*{Acknowledgements} 
It is our pleasure to acknowledge our great intellectual debt to Jacob Lurie.

 \subsection{Programmatic overview}
This work is part of a larger program, currently in progress. We now outline a part of this program, in order of logical dependency.
\begin{enumerate}

\item[\bf \cite{aft1}:] {\bf Local structures on stratified spaces} establishes a theory of stratified spaces based on the notion of conical smoothness.
This theory is tailored for the present program, and intended neither to supplant or even address outstanding theories of stratified spaces.
This theory of conically smooth stratified spaces and their moduli is, in effect, defined exactly so that it is closed under the basic operations of taking products, open cones of compact objects, restricting to open subspaces, and forming open covers, and it features a notion of \emph{derivative} which, in particular, gives the following: 
\begin{itemize}
\item[~] For the open cone $\sC(L)$ on a compact stratified space $L$, taking the derivative at the cone-point implements a homotopy equivalence between \emph{spaces} of conically smooth automorphisms
\[
\Aut\bigl(\sC(L)\bigr)~\simeq~\Aut(L)~.
\] 
\end{itemize}
This notion of derivative implies the existence of functorial resolutions singularities, by which we mean (spherical) blow-ups along closed unions of strata.

This work also introduces the notion of a constructible bundle, along with other classes of maps between stratified spaces.

\item[\bf \cite{striat}:]  {\bf A stratified homotopy hypothesis} establishes stratified spaces as parametrizing objects for $\infty$-categories.
Specifically, we construct a functor $\exit\colon \strat \to \Cat$ and show that the resulting restricted Yoneda functor $\Cat \to \Psh(\strat)$ is fully faithful.
The image is characterized by specific geometric descent conditions.
We call these presheaves \emph{striation sheaves}. 
We develop this theory so as to construct particular examples of $\infty$-categories by hand from stratified geometry: $\Bun$, $\Exit$, and variations thereof. 
As striation sheaves, $\Bun$ classifies constructible bundles, $\Bun \colon  K\mapsto \{X\xra{\sf cbl} K\}$,
while $\Exit$ classifies constructible bundles with a section.

\item[\bf \cite{emb1a}:]
{\bf Factorization homology I: higher categories} defines the $\infty$-categories $\cMfd_n^{\vfr}$ of \emph{vari-framed compact $n$-manifolds}, and $\cMfd_n^{\sfr}$ of \emph{solidly framed compact $n$-manifolds}. As a striation sheaf, $\cMfd_n^{\vfr}$ classifies proper constructible bundles equipped with a trivialization of their fiberwise constructible tangent bundle, and $\cMfd_n^{\sfr}$ classifies proper constructible bundles equipped with an injection of their fiberwise constructible tangent bundle into a trivial $n$-dimensional vector bundle.
We then construct a functor $\fC\colon (\cMfd_n^{\vfr})^{\op} \to \Cat_n$ between $\infty$-categories, and use this to define \emph{factorization homology}. This takes the form of a functor between $\infty$-categories
\[
\int \colon \Cat_n \longrightarrow \Fun(\cMfd_n^{\vfr}, \Spaces)
\]
that we show is fully faithful. 
In this sense, $(\infty,n)$-categories define coefficient systems for moduli of vari-framed stratifications of compact framed $n$-manifolds.

\item[\bf \cite{II}:] {\bf Factorization homology II: closed sheaves} establishes a version of factorization homology in where the higher categories are replaces by \emph{pointed} higher categories, and the manifolds are replaced by possibly non-compact manifolds. Further, it characterizes the closed sheaf conditions which identify the essential image of factorization homology homology theories. 

\item[\bf Conjectural:]
In analogy with \cite{emb1a} and \cite{II}, we conjecture that $(\infty,n)$-categories with adjoints define coefficient systems for \emph{solidly $n$-framed} compact $n$-manifolds.  
Namely, we conjecture the existence of a functor $\ov{\fC} \colon (\Mfd_n^{\sfr})^{\op} \to \Cat_n^{\ast/}$ between $\infty$-categories, which factors through $(\Cat_n^{\adj})^{\ast/}$, those pointed $(\infty,n)$-categories in which, for each $0<k<n$, each $k$-morphisms has both a left and a right adjoint.
From this, there is a theory of factorization homology \emph{with adjoints}, which gives a functor between $\infty$-categories
\[
\int \colon (\Cat_n^{\adj})^{\ast/} \longrightarrow \Fun(\Mfd_n^{\sfr},\Spaces)~,
\]
which we conjecture is fully faithful. In this sense, pointed $(\infty,n)$-categories with adjoints would define coefficients systems on the moduli of solidly $n$-framed stratifications of framed $n$-manifolds.

\item[\bf Present:]
The present work proves the cobordism hypothesis.  
Namely, for $\fX$ a symmetric monoidal $(\infty,n)$-category with adjoints and with duals, the $(\infty,n)$-category of $\fX$-valued fully extended framed topological quantum field theories is equivalent to the maximal $\infty$-subgroupoid of $\fX$: 
\[
\Fun^{\ot}(\Bord_n^{\fr}, \fX) ~\simeq~ \obj(\fX)~.
\]  
The cobordism hypothesis is a limiting consequence of the \emph{tangle hypothesis}, one form of which states that, for $\ast \xra{\uno} \cC$ a pointed $(\infty,n+k)$-category with adjoints, there is a canonical identification 
$\Fun_{\ast}(\cT{\sf ang}_{n\subset n+k}^{\fr}, \cC) \simeq \kEnd_\cC(\uno)$ between the space of pointed functors and the space of $k$-endomorphisms of the point in $\cC$.  
Through the conjectural factorization homology with adjoints, this tangle hypothesis amounts to showing that the pointed $(\infty,n+k)$-category $\cT{\sf ang}^{\fr}_{n\subset n+k}$ corresponds to the copresheaf on $\Mfd_{n+k}^{\sfr}$ represented by the object $\RR^k\in \Mfd_{n+k}^{\sfr}$. 
The bordism hypothesis follows from this tangle hypothesis, manifesting as the equivalence 
\[
\Bord_n^{\fr}~ \simeq~ \varinjlim \Omega^k \RR^k
\]
between copresheaves on $\Mfd_n^{\sfr}$.

\item[\bf Future:]
Using the current work, a future work will prove the following generalization of the cobordism hypothesis, articulated by Lurie in \cite{lurie.cobordism}, to manifolds equipped with more relaxed tangential structures than framings.
For $\fX$ a symmetric monoidal $(\oo,n)$-category with adjoints, there is a natural action of the orthogonal group $\sO(n)$ on the space of objects $\obj(\fX)$. 
For $F$ a space with an action of $\sO(n)$, there is an equivalence
\[
\Fun^{\ot}\bigl(\Bord_n^{F}, \fX\bigr) ~ \simeq ~  \Map^{\sO(n)}\bigl(F,\obj(\fX)\bigr)
\]
between the $(\infty,n)$-category of $\fX$-valued fully extended $F$-framed topological quantum field theories and the $\infty$-groupoid of $\sO(n)$-equivariant maps from $F$ to $\obj(\fX)$.

\end{enumerate}
 
 \section{Preliminaries}\label{Preliminaries}
 
 In this section, we recall definitions and results from our antecedent works. This section is only an overview; see those works for precise definitions and details.

\subsection{Stratified spaces}

The work \cite{aft1} presents a theory of stratified spaces founded on a key technical feature of conical smoothness. There, we define an $\oo$-category of (conically smooth) stratified spaces and (conically smooth stratified) maps among them. Here are a number of notable classes of morphisms in $\Strat$.
\begin{definition}\label{def.classes-of-maps}
Let $f\colon X\ra Y$ be a map of stratified spaces.
\begin{itemize}
\item {\bf Embedding (${\sf emb}$):}
$f$ is an \emph{open embedding} if it is an isomorphism onto its image as well as open map of underlying topological spaces.
\item {\bf Refinement (${\sf ref}$):}
$f$ is a \emph{refinement} if it is a homeomorphism of underlying topological spaces, and, for each stratum $X_p\subset X$, the restriction $f_|\colon X_p \to Y$ is an isomorphism onto its image.
\item {\bf Open (${\sf open}$):}
$f$ is \emph{open} if it is an open embedding of underlying topological spaces and a refinement onto its image.  

\item {\bf Fiber bundle:}
$f$ is a \emph{fiber bundle} if, the collection of images $\phi(O) \subset Y$, indexed by pullback diagrams
\[
\xymatrix{
F\times O \ar[r]  \ar[d]
&
X  \ar[d]
\\
O \ar[r]^-{\phi}
&
Y
}
\]
in which the horizontal maps are open embeddings, forms a basis for the topology of $Y$.

\item {\bf Constructible bundle ($\cbl$):}
 $f$ is a \emph{weakly constructible bundle} if, for each stratum $Y_q\subset Y$, the restriction $f_{|}\colon f^{-1}Y_q \to Y_q$ is a fiber bundle. The definition of a constructible bundle is inductive based on depth: in the base case of smooth manifolds, $f\colon X\ra Y$ is a constructible bundle if it is a fiber bundle; in the inductive step of the definition, $f\colon X\ra Y$ is a constructible bundle if it is a weakly constructible bundle and, additionally, if for each stratum $Y_q\subset Y$ the natural map
 \[
\Link_{f^{-1}Y_q}(X) \longrightarrow f^{-1}Y_q \underset{Y_q}\times \Link_{Y_q}(Y)
\]
is a constructible bundle.

\item {\bf Proper constructible ($\pcbl$):}
$f$ belongs to the class ($\pcbl$) if it is a constructible bundle and it is \emph{proper}, i.e., if $f^{-1}C\subset X$ is compact for each compact subspace $C\subset Y$. $f$ belongs to either of the classes $(\pcbl, {\sf surj})$ or $({{\pcbl}, {\sf inj}})$ if it is proper constructible as well as either surjective or injective, respectively. 

\end{itemize}

\end{definition}

The $p$-simplex $\Delta^p$ carries the \emph{standard stratification} by the totally ordered set $[p]$. This has an inductive definition, using the identification $\Delta^p\cong \oC(\Delta^{p-1})$ with the closed-cone on $(p-1)$-simplex, together with the likewise identification $[p] \cong [p-1]^{\tl}$, with the left-cone on the totally ordered set $[p-1]$. Note that this is a highly asymmetric stratification. The standardly stratified simplices assemble as a cosimplicial stratified spaces.
\begin{definition}\label{def.simplices}
The \emph{standard} cosimplical stratified space is the functor
\[
\Delta^\bullet \colon \bdelta\longrightarrow \Strat~,\qquad [p]\mapsto \Bigl(\Delta^p\ni t \mapsto {\sf Max}\{i\mid t_i\neq 0\}\in [p]\Bigr)
\]
with the values on morphisms standard.  
\end{definition}

\begin{definition}[$\Bun$ and $\Exit$]\label{def.Bun}
$\Bun$ is the presheaf on stratified spaces that classifies constructible bundles:
\[
\Bun\colon K \mapsto \bigl|\{X\underset{\cbl}{\xra{\pi}} K\}\bigr|~,
\]
the moduli space of constructible bundles over $K$.  
$\cBun$ is the subpresheaf on stratified spaces that classifies \emph{proper} constructible bundles:
\[
\cBun\colon K \mapsto \bigl|\{X\underset{\sf p.cbl}{\xra{\pi}} K\}\bigr|~,
\]
the moduli space of proper constructible bundles.  
\\
$\Exit$ is the presheaf on stratified spaces that classifies constructible bundles equipped with a section:
\[
\Exit\colon K\mapsto \Bigl|\bigl\{X \overset{\pi,~\cbl}{\underset{\sigma}\rightleftarrows} K \mid \sigma \pi = 1 \bigr\}\Bigr|~.
\]

\end{definition}

The cumulative result of the work of \cite{striat}, and of all the regularity around substrata ensured by conical smoothness, is the following.

\begin{theorem}[\cite{striat}]
The space-valued presheaves $\Bun$ and $\cBun$ and $\Exit$ on $\Strat$ restrict along the functor $\bdelta \xra{\Delta^\bullet} \Strat$ as $\oo$-categories (via the fully faithful embedding of $\oo$-categories into space-valued presheaves on $\bDelta$).
\end{theorem}

\begin{remark}
The $\infty$-subcategory $\cBun \subset \Bun$ consists exactly of the \emph{compact} stratified spaces, but it is \emph{not} full.
For instance, there is no morphism from $\emptyset$ to $S^1$ in $\cBun$, while there is a unique such morphism in $\Bun$. This morphism is represented by the constructible bundle $S^1\times\Delta^{1}\smallsetminus\{0\} \ra \Delta^1$, which is not proper.  

\end{remark}

Here we name several classes of morphisms in $\Bun$.  
To identify these morphisms, it is convenient to use that morphisms in the $\oo$-category $\Bun$ can be constructed as mapping cylinders on stratified maps in two different ways: as cylinders on open maps or as reversed cylinders on proper constructible maps.

\begin{theorem}\label{thm.class.maps}
There are monomorphisms
\[
\Strat^{\opn}  \xra{\sf Cylo}  \Bun  \xla{\sf Cylr} (\Strat^{\pcbl})^{\op}~,
\]
given on morphisms as the respective constructible bundles over $[0,1]$:
\[
{\sf Cylo}(U \underset{\opn}{\xra{~\gamma~}} Y)~:=~ X\times [0,1] \underset{X\times (0,1]} \amalg Y\times (0,1]
\qquad \text{ and }\qquad
{\sf Cylr}(K \underset{\pcbl}{\xla{~\pi~}} X)~:=~ K\underset{X\times \{0\}} \amalg X\times [0,1]~.
\]

\end{theorem}

We next isolate the following important classes of morphisms.

\begin{definition}\label{def.classes}
\begin{itemize}
\item[~]
\item The $\oo$-subcategory $\Bun^{\sf cls.crt}\subset \Bun$ of {\it closed-creation} morphisms is the image of $(\Strat^{\pcbl})^{\op}$.
\item The $\oo$-subcategory $\Bun^{\cls}\subset \Bun$ of {\it closed} morphisms is the image of $(\Strat^{{\pcbl}, {\sf inj}})^{\op}$. 

\item The $\infty$-subcategory $\Bun^{\sf ref}\subset \Bun$ of \emph{refinement} morphisms is the image of $\Strat^{\sf ref}$.

\item The $\infty$-subcategory $\Bun^{\sf emb}\subset \Bun$ of \emph{open embedding} morphisms is the image of $\Strat^{\emb}$.  

\end{itemize}

\end{definition}

\begin{remark}\label{rem.nofact}
Theorem \ref{thm.class.maps} captures a defining feature of $\Bun$: it is an $\oo$-category of stratified spaces in which one can compose open maps and opposites of proper constructible bundles. This composition is not formal. 
In particular, 
these two classes of morphisms do not offer a factorization system, for the following reason.
The link construction of~\cite{aft1} offers a factorization of each morphism $X\to \Delta^1$ in $\Bun$ as a closed-creation morphism followed by an open morphism:
\[
X_{0}\xra{\sf cls.crt} \Link_{X_0}(X)\xra{\sf open} X_1~.
\]
However, the space of such factorizations is not contractible.
Nevertheless, in the case that $X_0$ is a smooth manifold, with the further construing that the second factor be an embedding, this space is contractible: this is Lemma~\ref{tau.factor}.

\end{remark}

\subsection{Framings}
The tangent bundle of any smooth manifold $M$ is classified by a map
\[
\sT_M\colon  M \longrightarrow \Vect^{\sim}
\]
from the underlying $\oo$-groupoid of $M$ to the maximal $\oo$-subgroupoid $\Vect^{\sim}$ of (finite dimensional real) vector spaces and spaces of linear isomorphisms among them. 
The codomain $\infty$-groupoid can be canonically identified
\[
\coprod_{k\geq 0} {\sf BO}(k)  ~\simeq~\Vect^{\sim}
\]
as the coproduct of the classifying spaces of the groups $\sO(k)$, regarded as group-objects in spaces. 
There is a stratified generalization, the \emph{constructible tangent bundle}, established in \cite{emb1a}. For a stratified space $X$, there exists a functor
\[
\exit(X) \ra \Vect^{\sf inj}
\]
from the exit-path $\oo$-category of $X$ to the $\oo$-category $\Vect^{\sf inj}$ of (finite dimensional real) vector spaces and spaces of linear injections among them.
In the case that $X$ is unstratified, meaning that it is a smooth manifold, its exit path $\infty$-category is the underlying $\infty$-groupoid of this smooth manifold, and the constructible tangent bundle of $X$ agrees with the tangent bundle of this smooth manifold.

Generalizing the construction of the fiberwise tangent bundle on the total space of a fiber bundle is the \emph{fiberwise constructible tangent bundle} on the total stratified space of each constructible bundle.  
The following result yields the existence of this construction, through the universal example $\Exit \to \Bun$.

\begin{theorem}[\cite{emb1a}]
There exists a symmetric monoidal functor
\[
\sT\colon \Exit \longrightarrow  \Vect^{\sf inj}~,
\]
which is a localization with respect to closed morphisms and embedding morphisms.
\end{theorem}

Let $X\xra{\pi}K$ be a constructible bundle.
We obtain the fiberwise constructible tangent bundle of $\pi$ as follows.
There is a canonical pullback diagram
\[
\xymatrix{
\exit(X)    \ar[rr]   \ar[d]
&&
\Exit\ar[d]
\\
\exit(K)    \ar[rr]
&&
\Bun
}
\]
where the functor $\exit(X) \ra \Exit$ classifies the constructible bundle $X\underset{K}\times X \overset{\sf proj}{\underset{\sf diag}{~\rightleftarrows~}} X$.
The \emph{fiberwise constructible tangent bundle (of $\pi$)} is the composite functor
\[
\sT_\pi \colon \exit(X)\xra{X\underset{K}\times X \rightleftarrows X} \Exit \overset{\sT}\longrightarrow \Vect^{\sf inj}~.
\]

\begin{definition}\label{def.vfr}
A \emph{tangential structure} is an $\infty$-category $\tau \to \Vect^{\sf inj}$ over $\Vect^{\sf inj}$.
Let $\tau$ be a tangential structure.
A \emph{$\tau$-framing} of a constructible bundle $X\xra{\pi}K$ is a lift 
\[
\xymatrix{
&&
\tau  \ar[d]
\\
\exit(X) \ar@{-->}[urr]   \ar[rr]^-{\sT_\pi}
&&
\Vect^{\sf inj}
}
\]
between the fiberwise dimension bundle and the fiberwise constructible tangent bundle.
A \emph{$\tau$-framing} of a stratified space $X$ is a $\tau$-framing of the constructible bundle $X\to \ast$.

\end{definition}

\begin{definition}\label{def.Bun.tau}
For $\tau$ a tangential structure, the $\infty$-category of \emph{$\tau$-framed manifolds},
\[
\Bun^\tau~,
\]
is that classifying $\tau$-framed constructible bundles: it is an $\infty$-category over $\Bun$ for which, for each constructible bundle $X\xra{\pi}K$, the $\infty$-groupoid of lifts
\[
\xymatrix{
&&
\Bun^\tau  \ar[d]
\\
\exit(K) \ar@{-->}[urr]   \ar[rr]^-{X\xra{\pi}K}
&&
\Bun
}
\]
is identified as the $\infty$-groupoid of $\tau$-framings of $\pi$.  
The $\infty$-category of \emph{compact $\tau$-framed manifolds},
\[
\cBun^\tau~:=~\Bun^\tau_{|\cBun}~,
\]
is the base change of $\Bun^\tau$ along the monomorphism $\cBun \hookrightarrow \Bun$.

\end{definition}

\begin{theorem}[\cite{emb1a}]\label{bun.tau.exists}
For each tangential structure $\tau$, the $\infty$-category $\Bun^\tau$ over $\Bun$ exists and the $\infty$-category $\cBun^\tau$ over $\cBun$ exists.

\end{theorem}

\subsection{Flagged higher categories and Segal sheaves}

We now define Joyal's category $\btheta_n$, the $n$-fold wreath product of the simplex category $\bdelta$: see \cite{berger}, \cite{rezk-n}, and \S3.5 of \cite{emb1a}.
Denote $\Fin_\ast:=\Fin^{\ast/}$ for the category of based finite sets; each object in which may be denoted $I_+ := I\amalg \{+\}$ where $+$ is the base-point. 
The category  
\[
\Fin_{\ast \star}~\subset~ \Fin_\ast^{\star_+/}
\]
is the subcategory of the undercategory consisting of those based maps $(\star_+ \xra{f} I_+)$ from the two-element based set for which the inequality $f(\star)\neq +$ holds.  
The canonical projection $\Fin_{\ast \star} \to \Fin_\ast$ is an exponentiable fibration in the sense of~\cite{fibrations}.  
This observation validates the following.
 
\begin{definition}[Wreath]\label{def.wreath}
For $\cD\to \Fin_\ast$ an $\oo$-category over based finite sets, 
the \emph{wreath} functor 
\[
\cD\wr - \colon \Cat \longrightarrow \Cat_{/\cD}
\]
is right adjoint to the composite functor
\[
{\Cat}_{/\cD} \longrightarrow {\Cat}_{/\Fin_\ast} \xra{~{}~-\underset{\Fin_\ast}\times \Fin_{\ast \star}~{}~} {\Cat}_{/\Fin_{\ast\star}} \longrightarrow \Cat~.
\]
\end{definition}

Joyal's category $\bTheta_n$ is defined by induction on $n\geq 0$ by iterating the wreath construction with respect to the functor 
\[
\bDelta^{\op} \xra{~\Delta[1]/\partial \Delta[1]~} \Fin_\ast~.
\]
\begin{definition}[\cite{joyaltheta}]\label{def.theta}
For $n\geq 0$, Joyal's category $\btheta_n$ is as follows:
\begin{itemize}
\item The category $\bTheta_0:=\ast$ is terminal.
\item 
Suppose $n>0$.
The category $\bTheta_n^{\op}$ is the wreath product of $\bDelta^{\op}$ and $\bTheta_{n-1}^{\op}$: 
\[
\bTheta_n^{\op} ~:=~ \bDelta^{\op} \wr \bTheta_{n-1}^{\op}~.
\] 
\end{itemize}
\end{definition}

\begin{definition}\label{def.n-Segal-cov}
Let $n\geq 0$.  
\begin{itemize}

\item {\bf Inerts:}  
The \emph{inert} $\infty$-subcategory 
\[
\bTheta_{n}^{\sf inrt}~\subset~\btheta_n
\]
defined by induction on $n\geq 0$ as follows.
\begin{itemize}
\item 
$\bTheta_{0}^{\sf inrt}~\subset~\btheta_0 = \ast$.  

\item 
Suppose $n>0$.  
Consider the subcategory $\bDelta^{\sf inrt}\subset \Delta$ consisting of the consecutive monomorphisms.  
Define $(\btheta_n^{\sf inrt})^{\op}:=(\bdelta^{\sf inrt})^{\op} \wr (\btheta_{n-1}^{\sf inrt})^{\op}$.

\end{itemize}

\item {\bf Cells:}
Let $0\leq i \leq n$.
The \emph{$i$-cell} $c_i\in \btheta_n^{\op}$ is the initial object if $i=0$ and if $i>0$ it is the object $\ast \xra{\{c_i\}} \btheta_n^{\op}$ representing the pair of functors $\ast \xra{\{[1]\}} \bdelta^{\op}$ and $\ast \simeq \{[1]\}\underset{\Fin_\ast}\times \Fin_{\ast \star} \xra{\{c_{i-1}\}} \btheta_{n-1}^{\op}$.   
The \emph{boundary of the $i$-cell} is the subfunctor
\[
\partial c_i \subset c_i \colon \bTheta_n^{\op} \longrightarrow \Spaces
\]
of the representable functor, consisting, for each $T\in \bTheta_n$, of those morphisms $T\to c_i$ in $\bTheta_n$ that factor thorugh an object $S\in \bTheta_{i-1}\subset \bTheta_n$.

\item {\bf Segal covering diagrams:} A colimit diagram $[1]\times[1]\to \btheta_n$, written
\[
\xymatrix{
T_0  \ar[r]  \ar[d]
&
T'  \ar[d]
\\
T''  \ar[r]
&
T,
}
\]
is a \emph{Segal covering} diagram if each arrow is inert.   

\item {\bf Univalence diagrams:} For $n>0$, a colimit diagram $\cE^{\tr} \to \btheta_n$ is a \emph{univalence} diagram if it has either of the following two properties:
\begin{itemize}

\item The projection $\cE^{\tr} \to \bdelta$ factors through the functor $\ast \xra{\{c_1\}} \bdelta$, and the functor $(\cE^{\tr})^{\op} \simeq (\cE^{\tr})^{\op} \underset{\Fin_\ast}\times \Fin_{\ast \star} \to \btheta_{n-1}^{\op}$ is the opposite of a univalence diagram.

\item The projection ${\cE^{\tr}} \to \bdelta$ is the univalence diagram, and the functor $(\cE^{\tr})^{\op}\underset{\Fin_\ast}\times \Fin_{\ast \star} \to \btheta_{n-1}^{\op}$ factors through $\ast \xra{\{c_0\}} \btheta_{n-1}^{\op}$.

\end{itemize}

\item {\bf Zero-pointed:}
The category $\bTheta_{n,\emptyset}$ is initial among zero-pointed categories under $\bTheta_n$; it is obtained from $\bTheta_n$ by freely adjoining a zero-object.

\end{itemize}

\end{definition}

We will use the following definition of $(\oo,n)$-categories due to Rezk. Regarding other definitions of $(\oo,n)$-categories, see \cite{clark.chris}, \cite{bergnerrezk1}, \cite{simpson}, and the references therein.

\begin{definition}[\cite{rezk-n}]\label{def.segal.sheaves}
The $\oo$-category of \emph{Segal sheaves} on $\btheta_n$
\[
\Shv(\btheta_n)~\subset~ \Fun(\btheta_n^{\op},\spaces)
\]
is the full $\oo$-subcategory consisting of those functors that carry (the opposites of) Segal covering diagrams to limit diagrams.
The $\oo$-category of \emph{univalent Segal sheaves} on $\btheta_n$
\[
\Shv^{\sf unv}(\btheta_n)~\subset~  \Shv(\btheta_n)
\]
is the full $\oo$-subcategory consisting of those Segal sheaves that carry (the opposites of) univalence diagrams to limit diagrams. 
The $\oo$-category of \emph{$(\oo,n)$-categories} is the $\oo$-category of univalent Segal sheaves on $\btheta_n$:
\[
\Cat_n~: =~ \Shv^{\sf unv}(\btheta_n)~.
\]
\end{definition}

Because Segal covering diagrams and univalence diagrams are in particular colimit diagrams in $\bTheta_n$, the Yoneda functor factors:
\[
\bTheta_n~\hookrightarrow~\Cat_n~\subset~\Fun(\bTheta_n^{\op},\Spaces)~.
\]
We point out the colimit diagram in $\fCat_n$: for each $0\leq i \leq n$,
\begin{equation}\label{4}
\xymatrix{
\partial c_{i-1}    \ar[r]  \ar[d]
&
c_{i-1}  \ar[d]^-{t}
\\
c_{i-1}  \ar[r]^-{s}  
&
\partial c_i .
}
\end{equation}

\begin{remark}\label{rem.enriched}
A notion of $\Cat(\cV)$, $\oo$-categories enriched in a symmetric monoidal $\oo$-category $\cV$, is established in~\cite{gepnerhaugseng}. 
An alternative inductive definition for $(\oo,n)$-categories is as
\[
\Cat_n ~ \simeq~ \Cat(\Cat_{n-1})~,
\]
$\oo$-categories enriched in $(\oo,n-1)$-categories.
See Bergner--Rezk \cite{bergnerrezk2} for the equivalence of these notions.
\end{remark}

\begin{remark}
The more general $\oo$-category of Segal sheaves will also play an important role for us. 
Indeed, compact manifolds and cobordisms among them naturally organize as a Segal sheaf on $\btheta_n$ which is not univalent due to the existence of non-trivial h-cobordisms.
\end{remark}

\begin{definition}\label{def.flagged}
The $\oo$-category of \emph{flagged $(\oo,n)$-categories}is the full $\oo$-subcategory
\[
\fCat_n~\subset~ \Fun([n],\Cat_n)
\]
consisting of those $\cC_0\ra \cC_1\ra \ldots \ra \cC_n$ for which:
\begin{itemize}
\item for each $0\leq k \leq n$, the $(\infty,n)$-category $\cC_k$ is an $(\oo,k)$-category;
\item for each $0\leq i \leq j \leq k\leq n$, the functor $\cC_j \ra \cC_k$ is \emph{$(k-j)$-connective}, which is to say each solid diagram among $(\infty,n)$-categories
\[
\xymatrix{
\partial c_i \ar[d]\ar[r]&\cC_j\ar[d]\\
c_i\ar[r]\ar@{-->}[ur]&\cC_k}
\]
admits a filler.  
\end{itemize}
\end{definition}

\begin{observation}[\cite{flagged}]\label{cats.are.flagged}
The projection $\ev_n\colon \fCat_n \to \Cat_n$ is a localization, with right adjoint
\[
\Cat_n \longrightarrow \fCat_n~,\qquad
\cC~\mapsto~(\cC_{\leq 0} \to \cC_{\leq 1} \to \dots \to \cC_{\leq n})~,
\]
where, for each $0\leq i \leq n$,  $\cC_{\leq i}$ is the maximal $(\infty,i)$-subcategory of $\cC$.  

\end{observation}

\begin{theorem}[\cite{flagged}]\label{thm.flagged}
The restricted Yoneda functor $\fCat_n \ra \Fun(\btheta_n^{\op},\spaces)$ is fully faithful with image the Segal sheaves:
\[
\fCat_n ~ \simeq~  \Shv(\btheta_n)~.
\]
The restricted Yoneda functor $\fCat_n^{\ast/} \ra \Fun(\btheta_{n,\emptyset}^{\op},\spaces)$ is fully faithful with image the pointed Segal sheaves:
\[
\fCat_n^{\ast/} ~\simeq~ \Shv(\btheta_{n,\emptyset})
\]
between pointed flagged $(\oo,n)$-categories and those Segal sheaves $\cF$ on $\btheta_{n,\emptyset}$ whose value on the empty set $\emptyset$ is terminal: $\cF(\emptyset)=\ast$.

\end{theorem}

\begin{remark}
Theorem \ref{thm.flagged} is conceptually convenient in the present work, but it is not logically necessary. If one simply takes the equivalence $\fCat_n \simeq \Shv(\btheta_n)$ to be the definition of $\fCat_n$, then the present work becomes logically independent of \cite{flagged}.

\end{remark}

\begin{terminology}\label{source.target}
Let $0\leq k \leq n$.
Let $\cC\in \fCat_n$ be a flagged $(\infty,n)$-category.
The \emph{space of $k$-morphisms in $\cC$} is the value $\cC(c_k)$; the \emph{space of objects in $\cC$} is the value 
\[
\obj(\cC)~:=~ \cC(c_0)~.
\]
The \emph{source-target} map is the canonical map $(s,t)\colon \cC(c_k) \to \cC(\partial c_k)$;
the pushout square~(\ref{4}) determines the pullback square among spaces:
\[
\xymatrix{
\cC(\partial c_k)  \ar[r]^-{s}  \ar[d]_-{t}
&
\cC(c_{k-1})  \ar[d]^-{(s,t)}
\\
\cC(c_{k-1})  \ar[r]^-{(s,t)}
&
\cC(\partial c_{k-1}).
}
\]
For $(x,y)\colon \partial c_{k-1} \to \cC$ a pair of $(k-1)$-morphisms in $\cC$ with identified source-target, the \emph{space of $i$-morphisms in $\cC$ from $x$ to $y$} is the pullback among spaces:
\[
\xymatrix{
\Map_\cC(x,y)\ar[r]\ar[d]&\cC(c_k)\ar[d]
\\
\ast\ar[r]^-{\{x,y\}}&\cC(\partial c_k) .
}
\]
For $0\leq i\leq k$, the \emph{identities} map is the map $\cC(c_i) \xra{\sf id } \cC(c_k)$ obtained by evaluating on the canonical map $c_k \to c_i$.  
For each object $x\in \cC(c_0)$, the \emph{space of $k$-endomorphisms of $x$ in $\cC$} is the pullback among spaces:
\[
\xymatrix{
\kEnd_\cC(x)   \ar[r]  \ar[d]
&
\cC(c_k)   \ar[d]^-{(s,t)}
\\
\{{\sf id}_{x}\}     \ar[r]
&
\cC(\partial c_k).
}
\]

\end{terminology}

\subsection{Factorization homology}
We explain our primary conduit between higher categories and differential topology: factorization homology, as developed in~\cite{emb1a}.
This form of factorization homology this connects $(\infty,n)$-categories with vari-framed stratified spaces of dimension at most $n$.  
We begin with the notion of \emph{vari-framing}.

Observe the functor from the poset of non-negative integers,
\[
\epsilon \colon \ZZ_{\geq 0} \longrightarrow \Vect^{\sf inj}
~,\qquad
\bigl(\RR^0\hookrightarrow \RR^1  \hookrightarrow \dots\bigr)~,
\]
classifying the injections as first coordinates among Euclidean spaces.
For $n\geq 0$, observe the composite functor
\begin{equation}\label{n.to.vect}
[n] = (\ZZ_{\geq 0})_{/n}\longrightarrow \ZZ_{\geq 0} \longrightarrow \Vect^{\sf inj}
~,\qquad
\bigl(\RR^0\hookrightarrow \RR^1 \hookrightarrow \dots\hookrightarrow \RR^n\bigr)~.
\end{equation}
\begin{definition}\label{def.mfd.vfr}
For $n\geq 0$, the $\infty$-categories of \emph{vari-framed $n$-manifolds} and of \emph{vari-framed compact $n$-manifolds} are
\[
\Mfd^{\vfr}_n~:=~\Bun^{[n]}
\qquad\text{ and }\qquad
\cMfd_n^{\vfr}~:=~\cBun^{[n]}~,
\qquad \text{ respectively },
\]
where $[n]$ is regarded as a tangential structure via~(\ref{n.to.vect}).

\end{definition}

\begin{remark}
Assigning to each point in a stratified space the dimension of the stratum in which it lies assembles as a functor
\[
{\sf dim}\colon \Exit \longrightarrow \ZZ_{\geq 0}
~,\qquad
(x\in X)\mapsto {\sf dim}_x(X)~.
\]
For $X\xra{\pi}K$ a constructible bundle, the \emph{fiberwise dimension} bundle over $X$ is the functor
\[
\epsilon_\pi \colon \exit(X)\xra{X\underset{K}\times X \rightleftarrows X} \Exit \overset{\sf dim}\longrightarrow \ZZ_{\geq 0} \xra{~\epsilon~} \Vect^{\sf inj}~.
\]
Unwinding Definition~\ref{def.mfd.vfr}, a \emph{vari-framing} of a constructible bundle $X\xra{\pi}K$ is an identification
\[
\varphi\colon \epsilon_\pi~\simeq~ \sT_\pi
\]
between the fiberwise dimension bundle and the fiberwise constructible tangent bundle.
In particular, \emph{vari-framing} of a stratified space $X$ is an identification 
\[
\varphi\colon \epsilon_X \simeq \sT_X
\]
between its constructible tangent bundle and its dimension constructible bundle.

\end{remark}

\begin{example}
For each $k\geq 0$, the \emph{hemispherically stratified $k$-disk} $\DD^k$ is equipped with a standard vari-framing.

\end{example}

\begin{theorem}[\cite{emb1a}]\label{theta.ff}
For each $n\geq 0$, there exist fully faithful functors
\begin{equation}
\btheta_n^{\op} \xra{\langle-\rangle} \cMfd_n^{\vfr}
 \qquad {\rm and} \qquad \btheta_{n,\emptyset}^{\op} \xra{\langle-\rangle} \Mfd_n^{\vfr}
\end{equation}
from Joyal's category and its based counterpart $\btheta_{n,\emptyset}$.
For each $0\leq k \leq n$, these functors evaluate as $\lag c_k\rag = \DD^k$.

\end{theorem}

\begin{definition}[Closed sheaves]\label{def.closed-cover}
Let $\tau$ be a tangential structure.
A limit diagram 
\[
\xymatrix{
X   \ar[r]   \ar[d] 
&
X''   \ar[d]
\\
X'   \ar[r]
&
X_0,
}
\]
in $\Bun^\tau$ is
\begin{itemize}
\item a \emph{purely closed cover} if the diagram is comprised of closed morphisms;
\item a \emph{refinement-closed cover} if the horizontal arrows are refinement morphisms while the vertical arrows are closed morphisms.
\item an \emph{embedding-closed cover} if the horizontal arrows are embedding morphisms while the vertical arrows are closed morphisms.
\end{itemize}
The limit diagram is a {\it closed cover} if it is either a purely closed cover or a refinement-closed cover. 
For $\cD^\tau\subset \Bun^\tau$ an $\infty$-subcategory, a functor from $\cD^\tau$ is a \emph{closed sheaf} if it carries closed covers to limit diagrams; the $\infty$-category of \emph{closed sheaves} on $\cD^\tau$ is the full $\infty$-subcategory
\[
\cShv(\cD^\tau)~\subset~\Fun(\cD^\tau , \Spaces)
\]
consisting of the closed sheaves.

\end{definition}

\begin{definition}\label{def.disk.vfr}
For each $n\geq 0$, the $\infty$-categories of \emph{disk-stratified vari-framed $n$-manifolds} and of \emph{disk-stratified compact vari-framed $n$-manifolds} are, respectively, the smallest full $\infty$-subcategories
\[
\bTheta_{n,\emptyset}^{\op}~\subset~\Disk_n^{\vfr}~\subset~\Mfd_n^{\vfr}
\qquad \text{ and } \qquad
\bTheta_n^{\op}~\subset~\cDisk_n^{\vfr}~\subset ~\cMfd_n^{\vfr}
\]
that are closed under the formation of all closed covers.

\end{definition}

Restriction and right Kan extension along the fully faithful functors in Definition~\ref{def.disk.vfr} define adjunctions
\begin{equation}\label{adjunctions}
\Fun(\Disk_n^{\vfr},\Spaces) ~\rightleftarrows~ \Fun(\bTheta_{n,\emptyset}^{\op},\Spaces)
\qquad\text{ and }\qquad
\Fun(\cDisk_n^{\vfr},\Spaces) ~\rightleftarrows~ \Fun(\bTheta_n^{\op},\Spaces)~.
\end{equation}

The following is a main result of~\cite{II}.  
\begin{theorem}[\cite{II}]\label{cats.are.sheaves}
For each $n\geq 0$, the adjunctions in~(\ref{adjunctions}) restricts as equivalences between $\infty$-categories:
\[
\cShv(\Disk_n^{\vfr})~\simeq~\fCat_n^{\ast/}
\qquad\text{ and }\qquad
\cShv(\cDisk_n^{\vfr})~\simeq~\fCat_n~.
\]
\end{theorem}

Factorization homology on objects of $\cMfd_n^{\vfr}$ with coefficients in an $(\oo,n)$-category $\cC$ is then given by the left Kan extension from $\cDisk_n^{\vfr}$ to $\cMfd_n^{\vfr}$ of the closed sheaf associated to $\cC$.
A choice of object in $\cC$ determines an extension of factorization homology over possibly non-compact vari-framed $n$-manifolds.

\begin{definition}\label{def.fact.hom}
For each $n\geq 0$, \emph{factorization homology (for $(\infty,n)$-categories)} and \emph{factorization homology (for pointed $(\infty,n)$-categories)} is, respectively, right Kan extension followed by left Kan extension:
\[
\int\colon \fCat_n~ \overset{\sf RKan}\simeq~\cShv(\cDisk_n^{\vfr})\xra{~\sf LKan~} \Fun(\cMfd_n^{\vfr},\spaces)
\]
and
\[
\int\colon \fCat_n^{\ast/}~ \overset{\sf RKan}\simeq~\cShv(\Disk_n^{\vfr})\xra{~\sf LKan~} \Fun(\Mfd_n^{\vfr},\spaces)~.
\]
\end{definition}

\begin{theorem}[\cite{II}]\label{on.Rk}
The factorization homology functors $\int$ of Definition~\ref{def.fact.hom} are fully faithful, and the diagram among $\infty$-categories
\[
\xymatrix{
\fCat_n^{\ast/}  \ar[rr]^-{\int}  \ar[d]
&&
\Fun(\cMfd_n^{\vfr},\spaces)    \ar[d]
\\
\fCat_n  \ar[rr]^-{\int}
&&
\Fun(\cMfd_n^{\vfr},\spaces)
}
\]
canonically commutes.
For each pointed flagged $(\oo,n)$-category $\uno \in \cC$, 
its factorization homology $\int \cC$ evaluates as
\begin{equation}
\int_{\DD^k}\cC~ \simeq~ \cC(c_k) \qquad {\rm and} \qquad \int_{\RR^k}\cC~ \simeq~ \kEnd_{\cC}(\uno)
\end{equation}
where $\cC(c_k)$ is the space of $k$-morphisms of $\cC$ and $\kEnd_{\cC}(\uno)$ is the space of $k$-endomorphisms of the distinguished object $\uno\in \cC$ from Terminology~\ref{source.target}.
\end{theorem}

\subsection{Factorization homology with adjoints}\label{sec.adj}

We expect a similar story to that described in the preceding sections involving $(\oo,n)$-categories \emph{with adjoints} and a form of factorization homology based on \emph{solid $n$-framings}. 
To define solid $n$-framings, we use the forgetful functor
\begin{equation}\label{sfr}
\Vect^{\sf inj}_{/\RR^n} \ra \Vect^{\sf inj}
\end{equation}
from the $\oo$-overcategory of $\RR^n$.

\begin{definition}\label{def.sfr}
For $n\geq 0$, the respective $\infty$-categories of \emph{solidly $n$-framed manifolds} and of \emph{solidly $n$-framed compact manifolds} are
\[
\Mfd^{\sfr}_n~:=~\Bun^{\Vect^{\sf inj}_{/\RR^n}}
\qquad\text{ and }\qquad
\cMfd_n^{\sfr}~:=~\cBun^{\Vect^{\sf inj}_{/\RR^n}}~,
\]
where ${\Vect^{\sf inj}_{/\RR^n}}$ is regarded as a tangential structure via~(\ref{sfr}).

\end{definition}

For each $n\geq 0$, the evident functor $[n]\to \Vect^{\sf inj}_{/\RR^n}$ over $\Vect^{\sf inj}$ determines functors among $\infty$-categories:
\begin{equation}\label{2}
\Mfd^{\vfr}_n\longrightarrow  \Mfd^{\sfr}_n
\qquad\text{ and }\qquad
\cMfd^{\vfr}_n  \longrightarrow  \cMfd^{\sfr}_n~.
\end{equation}

\begin{definition}[Adjoints]\label{def.adjoints}
Let $\cC$ be an $(\infty,n)$-category.  
Let $c_k\xra{L} \cC$ select a $k$-morphism in $\cC$, written through Terminology~\ref{source.target} as $x\xra{L} y$, where $x$ and $y$ are $(k-1)$-morphisms with common source-target.
This $k$-morphism $L$ is a \emph{left adjoint} if there is a $k$-morphism $y\xra{R}x$ in $\cC$ (the \emph{right adjoint}) for which the following exists.
\begin{itemize}
\item A \emph{unit} $(k+1)$-morphism in $\cC$ and a \emph{counit} $(k+1)$-morphism in $\cC$: 
\[
\eta\colon {\sf id}_x \longrightarrow R\circ L
\qquad \text{ and }\qquad
\epsilon \colon L\circ R \longrightarrow {\sf id}_y~.
\]

\item 
Identifications of the composite $(k+1)$-morphisms in $\cC$:
\[
{\sf id}_L \colon L\xra{~L(\eta)~} L\circ R \circ L \xra{~\epsilon_L~}L
\qquad\text{ and }\qquad
{\sf id}_R\colon R \xra{~\eta_R~} R\circ L \circ R \xra{~R(\epsilon)~} R~.
\]

\end{itemize}
The $\infty$-category of \emph{$(\infty,n)$-categories with adjoints} is the full $\infty$-subcategory
\[
\Cat_n^{\adj}~\subset~ \Cat_n
\]
consisting of those $(\oo,n)$-categories in which, for each $0<k<n$, each $k$-morphism is both a left adjoint and a right adjoint.
\end{definition}

In this work we assume the following conjecture, which asserts a theory of factorization homology with adjoints.

\begin{conj}\label{conj.one.again}
There exists a fully faithful functor $\int\colon (\Cat_n^{\sf adj})^{\ast/}\hookrightarrow\Fun(\Mfd_n^{\sfr}, \Spaces)$
from pointed $(\oo,n)$-categories with adjoints to space-valued functors on solidly $n$-framed stratified manifolds with respect to which the diagram
\[
\xymatrix{
(\Cat_n^{\sf adj})^{\ast/}\ar[rr]^-{\int}\ar[d]&&\Fun(\Mfd_n^{\sfr}, \Spaces)\ar[d]\\
\Cat_n^{\ast/}\ar[rr]^-\int &&\Fun(\Mfd_n^{\vfr},\Spaces)}
\]
canonically commutes.
\end{conj}
We refer to both of these horizontal functors as \emph{factorization homology}.

\section{The tangle hypothesis}

Fix $0<k\leq n$.
In this section we construct the pointed flagged $(\infty,n+k)$-category
\[
\Tang_{n\subset n+k}^{\fr}
\]
of framed codimension-$k$ tangles.
We construct $\Tang_{n\subset n+k}^{\fr}$ in the following steps.
\begin{enumerate}
\item As Definition~\ref{def.fr.tang}, we construct an $\oo$-category $\fTang_{n\subset n+k}^{\fr}$ over $\Mfd_{n+k}^{\sfr}$ for which an object over $D\in \Mfd_{n+k}^{\sfr}$ is a codimension-$k$ tangle in $D$ with framing data.
\item In Corollary \ref{cor.actuallyleft}, we show the functor of $\oo$-categories $\fTang_{n\subset n+k}^{\fr}\ra \Mfd_{n+k}^{\sfr}$ is a left fibration; its straightening is then a functor
\[
\Tang_{n\subset n+k}^{\fr}\colon \Mfd_{n+k}^{\sfr} \longrightarrow\Spaces
~,\qquad
D\mapsto (\fTang_{n\subset n+k}^{\fr})_{|D}~.
\]
To do so, we establish an equivalence over $\Mfd_{n+k}^{\sfr}$:
\[
\fTang_{n\subset n+k}^{\fr}~ \simeq~ (\Mfd_{n+k}^{\sfr})^{\RR^k/}~.
\]
This equivalence identifies the functor $\Tang_{n\subset n+k}^{\fr}\colon \Mfd_{n+k}^{\sfr}\ra \Spaces$ as that corepresented by the object $\RR^k\in \Mfd_{n+k}^{\sfr}$~.
\item In Corollary \ref{tang.cat}, we show that the space-valued functor $\Tang_{n\subset n+k}^{\fr}$ is a closed sheaf.
It follows that the restriction of the closed sheaf $\Tang_{n\subset n+k}^{\fr}$ along $\btheta^{\op}_{n+k,\emptyset} \hookrightarrow \Mfd_{n+k}^{\vfr}\ra \Mfd_{n+k}^{\sfr}$ is a Segal sheaf on $\btheta^{\op}_{n+k,\emptyset}$, which is a pointed flagged $(\oo,n+k)$-category.
Through~\S\ref{sec.adj}, this pointed flagged $(\infty,n+k)$-category has adjoints.
The tangle hypothesis is a direct compilation of the above points.
\end{enumerate}

\subsection{The $\oo$-category of tangles $\fTang_{n\subset n+k}^{\fr}$}

Consider the standard linear inclusion $\RR^k\subset\RR^{n+k}$ defined by selecting the first $k$ coordinates. This defines a solid $(n+k)$-framing on $\RR^k$.

For the next definition, we denote the full $\infty$-subcategories
\[
\Ar^{\sf cls.crt}(\Mfd_{n+k}^{\sfr})~\subset~\Ar(\Mfd_{n+k}^{\sfr})~\supset~\Ar^{\sf emb}(\Mfd_{n+k}^{\sfr})
\]
consisting respectively of the \emph{closed-creation} arrows in $\Mfd_{n+k}^{\sfr}$ and of the \emph{embedding} arrows in $\Mfd_{n+k}^{\sfr}$, the latter which are understood as open stratified embeddings.
\begin{definition}\label{def.fr.tang}
The $\infty$-category of \emph{framed codimension-$k$ tangles} is the $\infty$-category 
\[
\fTang_{n\subset n+k}^{\fr} 
\longrightarrow
\Mfd_{n+k}^{\sfr}
\]
as in the pullback diagram among $\infty$-categories:
\[
\xymatrix{
\fTang_{n\subset n+k}^{\fr}  \ar[rr]  \ar[d]
&
&
\Ar^{\sf cls.crt}(\Mfd_{n+k}^{\sfr})\underset{\Mfd_{n+k}^{\sfr}}\times \Ar^{\sf emb}(\Mfd_{n+k}^{\sfr})  \ar[d]^-{\ev_s\times \ev_t}
\\
\Mfd_{n+k}^{\sfr} \ar[rr]^-{(\{\RR^k\} , {\sf id} )}
&
&
\Mfd_{n+k}^{\sfr}  \times \Mfd_{n+k}^{\sfr}   .
}
\]
\end{definition}

Equivalently, the $\oo$-category of tangles is
\[
\fTang_{n\subset n+k}^{\fr} ~:=~
(\Mfd_{n+k}^{\sfr})^{\RR^k/^{\sf cls.crt}}\underset{\Mfd_{n+k}^{\sfr}}\times \Ar^{\sf emb}(\Mfd_{n+k}^{\sfr}) ~,
\]
where $(\Mfd_{n+k}^{\sfr})^{\RR^k/^{\sf cls.crt}} \subset (\Mfd_{n+k}^{\sfr})^{\RR^k/}$ is the full $\oo$-subcategory of the $\oo$-undercategory consisting of the closed-creation morphisms from $\RR^k$.
In the next remark, we justify the terminology of tangles in this definition.

\begin{remark}\label{rem.clscrtEuc}
By definition, $\fTang_{n\subset n+k}^{\fr}$ is the $\oo$-category of composable morphisms
\[
\RR^k\xra{\sf cls.crt} N \xra{\emb} D
\]
in $\Mfd_{n+k}^{\sfr}$, in which the first morphism is a closed-creation and the second morphism is an open stratified embedding. 
The functor $\fTang_{n\subset n+k}^{\fr}\ra \Mfd_{n+k}^{\sfr}$ carries such an object $\{\RR^k\ra N \ra D\}$ to its target $D$.
Now, as established in~\S6 of~\cite{striat}, each closed-creation morphism $f\colon M\to N$ in $\Bun$ from a smooth manifold $M$ is the reversed mapping cylinder on a proper stratified fiber bundle $M \xla{\pi} N$.
In the case that $M=\RR^k$, such a fiber bundle is canonically trivial:
\[
(\RR^k\xra{f} N)~\simeq~ {\sf Cylr}(\RR^k \xla{~\pr~} \RR^k\times W)~,
\qquad
\text{ where }W:=\pi^{-1}(0)~.
\]
Consequently, an object in $\fTang_{n\subset n+k}^{\fr}$ over $D\in \Mfd_{n+k}^{\sfr}$ is a pair
$(W,\RR^k\times W \hookrightarrow D)$ consisting of a compact stratified space $W$ equipped with a solid $n$-framing, together with a solidly $(n+k)$-framed open stratified embedding $\RR^k\times W\hookrightarrow D$.  
The data of this $(n+k)$-framed open stratified embedding forgets to that of a so-called \emph{framed codimension-$k$ tangle in $D$}:
\begin{itemize}
\item A stratified embedding $W\hookrightarrow D$ satisfying the following conditions.
\begin{itemize}
\item The stratified space $W$ is compact.
\item The stratification of $W$ is pulled back along this embedding from that of $D$.
\item The embedding admits a tubular neighborhood of rank $k$.  
This is to say that the embedding extends to an open embedding from the total space of a rank $k$ conically smooth vector bundle over $W$.
\end{itemize}
\item A splitting of the $(n+k)$-framing of $D$ along $W$.
This is the data of a lift among functors from $\exit(W)$ to $\Vect^{\sf inj}$:
\[
\xymatrix{
\sT_W  \ar@{-->}[rr]  \ar[d]
&&
\epsilon^n_W  \ar[d]^-{\sf last}
\\
(\sT_D)_{|W}  \ar[rr]^-{\varphi_{|W}}
&&
\epsilon^{k+n}_{W},
}
\]
where $\varphi$ is the given solid $(n+k)$-framing of $D$.
\end{itemize}
Because the space of tubular neighborhoods of a stratified subspace $W\subset D$ is empty or contractible, and because the space of solid framings pullback along open embeddings, the data of the pair $(W,\RR^k\times W \hookrightarrow D)$ is equivalent to that of a framed codimension-$k$ tangle in $D$.
In this way, for each solidly $(n+k)$-framed stratified space $D$, we interpret the fiber space of the projection $\fTang_{n\subset n+k}^{\fr} \to \Mfd_{n+k}^{\sfr}$ over $D$ as a moduli space
\[
(\fTang_{n\subset n+k}^{\fr} )_{|D}~\simeq~ \Bigl| \Bigl\{W\overset{\rm framed}{\underset{{\rm codim\text{-}}k}\hookrightarrow} D  \Bigr\} \Bigr|
\]
of framed codimension-$k$ tangles in $D$.

\end{remark}

\subsection{Constructible bundles $X\ra \Delta^1$ whose special fiber $X_0$ is a manifold}
The following lemma is the desired result of this section, which we will prove in the more general form of Lemma~\ref{tau.factor}.

\begin{lemma}\label{lemma.corep}
Composition defines an equivalence between $\oo$-categories over $\Mfd_{n+k}^{\sfr}$:
\[
\fTang_{n\subset n+k}^{\fr} \xra{~\simeq~} (\Mfd_{n+k}^{\sfr})^{\RR^k/}
~,\qquad
(\RR^k\underset{\sf cls.crt}{\xra{~\pi~}} N \underset{\sf emb}{\xra{~\gamma~}} D)\mapsto (\RR^k\xra{\gamma\circ \pi} D)~.
\]

\end{lemma}
\begin{proof}
Apply Lemma \ref{tau.factor} in the special case $\tau=\sfr_n$ and $M= \RR^k$.
\end{proof}

For $\tau$ a tangential structure, composition in $\Bun^\tau$ defines a functor over $\Bun^\tau\times \Bun^\tau$:
\[
\xymatrix{
\Ar^{\sf cls.crt}(\Bun^\tau)     \underset{\Bun^\tau} \times \Ar^{\sf emb}(\Bun^\tau)     \ar[rd]_-{\ev_s\times \ev_t}   \ar[rr]^-{\sf comp}
&&
\Ar(\Bun^\tau)      \ar[dl]^-{\ev_s\times \ev_t}
\\
&
\Bun^\tau\times \Bun^\tau
&
.
}
\]
For each object $M\in \Bun^\tau$, base change of this diagram along $\Bun^\tau \xra{(\{M\} , {\sf id})} \Bun^\tau \times \Bun^\tau$ gives a functor over $\Bun^\tau$:
\begin{equation}\label{3}
\xymatrix{
(\Bun^\tau)^{M/^{\sf cls.crt}}    \underset{\Bun^\tau} \times \Ar^{\sf emb}(\Bun^\tau)    \ar[rd]_-{\ev_t}   \ar[rr]^-{{\sf comp}}
&&
(\Bun^\tau)^{M/}    \ar[dl]^-{ \ev_t}
\\
&
\Bun^\tau
&
.
}
\end{equation}

\begin{definition}\label{def.stable.tau}
A tangential structure $\tau$ is \emph{stable} if the functor $\tau \to \Vect^{\sf inj}$ is a \emph{Cartesian fibration}.

\end{definition}

\begin{example}
Let $n\geq 0$.
Being the projection from an $\infty$-overcategory, the tangential structure $\Vect^{\sf inj}_{/\RR^n}$ is stable.
In contrast, the tangential structure $[n]$ is not stable (unless 
$n\leq 1$
).

\end{example}

\begin{observation}\label{stable.ref.loc}
Let $\tau$ be a tangential structure.  
Let $X_0 \xra{\gamma} X_1$ be an open morphism in $\Bun$, and let $\w{X}_1\in \Bun^\tau$ be a lift of the target of $\gamma$.
Provided $\tau$ is stable, there exists a Cartesian lift of $\gamma$ whose target is $\w{X}_1$.

\end{observation}

The next result shows that in a simple situation, the formation of links implements a factorization system by closed-creation morphisms followed by embedding morphisms.  

\begin{lemma}\label{tau.factor}
Let $\tau$ be a stable tangential structure.
Let $M$ be a smooth manifold equipped with a $\tau$-framing.
The composition functor in~(\ref{3}) is an equivalence between $\infty$-categories.
In particular, each morphism $M\to D$ in $\Bun^\tau$ admits a unique factorization
\[
M\xra{~\sf cls.crt~} E \xra{~\emb~} D
\]
as the reversed cylinder on a proper constructible bundle followed by a stratified open embedding morphism.  

\end{lemma}

\begin{proof}
The composition functor~(\ref{3}) is an equivalence if and only if, for each $[p]\in \bDelta$, the map between spaces of functors from $[p]$,
\[
\Cat\bigl( [p] ,
(\Bun^\tau)^{M/^{\sf cls.crt}} \underset{\Bun^\tau}\times \Ar^{\emb}(\Bun^\tau)
\bigr)
\longrightarrow
\Cat\bigl([p] , (\Bun^\tau)^{M/} \bigr)~,
\]
is an equivalence.
For this, it is enough to show, each point in the domain of this map, and for each $q\geq 0$, that the map between sets of connected components of based maps from the $q$-sphere,
\[
\pi_0 \Map^{\ast/}\Bigl(  S^q ,
\Cat\bigl( [p] ,
(\Bun^\tau)^{M/^{\sf cls.crt}} \underset{\Bun^\tau}\times \Ar^{\emb}(\Bun^\tau)
\bigr)
\Bigr)
\longrightarrow
\pi_0  \Map^{\ast/}\Bigl(  S^q ,
\Cat\bigl([p] , (\Bun^\tau)^{M/} \bigr)
\Bigr)~,
\]
is bijective.
What we show below is that the map between connected components of non-based maps
\begin{equation}\label{333}
\pi_0 \Map\Bigl(  S^q ,
\Cat\bigl( [p] ,
(\Bun^\tau)^{M/^{\sf cls.crt}} \underset{\Bun^\tau}\times \Ar^{\emb}(\Bun^\tau)
\bigr)
\Bigr)
\longrightarrow
\pi_0  \Map\Bigl(  S^q ,
\Cat\bigl([p] , (\Bun^\tau)^{M/} \bigr)~,
\Bigr)
\end{equation}
is bijective.
The case of based maps follows nearly identically; we leave the modification to the reader.
By adjunction, as the case $K=\Delta^p\times S^q$, bijectivity of~(\ref{333}) is implied by bijectivity of the map between sets of connected components of functors,
\begin{equation}\label{334}
\pi_0
\Cat\bigl( \exit(K) ,
(\Bun^\tau)^{M/^{\sf cls.crt}} \underset{\Bun^\tau}\times \Ar^{\emb}(\Bun^\tau)
\bigr)
\longrightarrow
\pi_0
\Cat\bigl( \exit(K) , (\Bun^\tau)^{M/} \bigr)~,
\end{equation}
for each compact stratified space $K$.
In this way, we are reduced to showing that this map~(\ref{334}) is bijective.
Unwinding, this is to show the existence of an essentially unique lift
\begin{equation}\label{lift}
\xymatrix{
&(\Bun^\tau)^{M/^{\sf cls.crt}} \underset{\Bun^\tau}\times \Ar^{\emb}(\Bun^\tau)\ar[d]\\
\exit(K)\ar@{-->}[ur]\ar[r]&(\Bun^\tau)^{M/}}
\end{equation}
for every compact stratified space $K$.

Using that $K$ is compact, the closed cone $\oC(K)= \{0\}   \underset{K\times \{0\}}\amalg      K\times\Delta^1$ is another compact stratified space.
By adjunction, a functor $\exit(K)\ra (\Bun^\tau)^{M/}$ is equivalent to a functor
\[
\exit(K)^{\tl}~\simeq ~\exit\bigl(\ov\sC(K)\bigr)\longrightarrow \Bun^\tau
\]
together with an identification of the value on the cone-point as $M$.
The latter classifies the following data:
\begin{itemize}
\item
a constructible bundle $X \xra{\pi} \oC(K)$;

\item
a $\tau$-structure on $\pi$, by which we mean a factorization $\sT_\pi \colon \exit(X) \to \tau \to \Vect^{\sf inj}$ of the fiberwise tangent classifier;

\item
an identification over $\{0\}\subset \oC(K)$, by which we mean an equivalence $M \simeq X_0$ over $\tau$ with the fiber over $\{0\}\subset \oC(K)$.

\end{itemize}

So fix such data.  
The lift~(\ref{lift}) classifies the following data:
\begin{itemize}
\item a constructible bundle $\widetilde{X}\ra \oC(K\times \Delta^1)$ over the closed cone,

\item fiberwise $\tau$-structure on it,

\item an identification 
over $\tau$ of its restriction:
\[
\xymatrix{
X       \ar[d]      \ar@{=}[rr]
&&
\widetilde{X}_{|\oC(K\times\{1\})}   \ar[d]    
\\
\oC(K)   \ar@{=}[rr]
&& 
\oC(K\times\{1\}) ;
}
\]
\end{itemize}
this data is subject to the conditions that
\begin{itemize}
\item 
for each point $k\in K$, the base change constructible bundle 
\[
\widetilde{X}_{|\oC(\{k\}\times\{0\})} 
\longrightarrow 
\oC\bigl(\{k\}\times\{0\}\bigr) = \oC\bigl(\{k\}\bigr) \cong \Delta^1
\]
is classified by a closed-creation morphism in $\Bun^\tau$,
\item 
for each point $k\in K$, the base change constructible bundle 
\[
\widetilde{X}_{|\{k\}\times \Delta^1} \longrightarrow \{k\}\times \Delta^1 = \Delta^1
\]
is classified by an embedding morphism in $\Bun^\tau$.
\end{itemize}
We first constructsuch a lift~(\ref{lift}). 
Take the stratified space $\Link_{X_0}(X)$, 
which is the link of the closed constructible subspace $X_0\subset X$. This link lies in a commutative diagram among stratified spaces:
\[
\xymatrix{
M=X_0 \ar[d]&\ar[d]\ar[l]_-{\sf p.cbl} \Link_{X_0}(X)\ar[r]^-{\sf open} &X_{|K}\ar[d]\\
\{0\}&\ar[l]
\Link_{\{0\}}\bigl(\oC(K)\bigr)   \ar[r]^-{\cong} 
& 
K .
}
\]
From the definition of links, the fact that $M$ is trivially stratified results in the following two simplifications.
\begin{enumerate}
\item
The proper constructible bundle $\Link_{X_0}(X)\ra X_{|K\times\Delta^1}$ is in fact a fiber bundle of stratified spaces.
(Heuristically, this is to say that the fibers vary continuously.)

\item
The open map $\Link_{X_0}(X)\ra X_{|K\times\Delta^1}$ is in fact an open stratified embedding.
(Heuristically, this is to say that the refinement map onto its image is in fact an isomorphism.)

\end{enumerate}
Taking fiberwise reversed mapping cylinder over $K$ gives a constructible bundle
\[
{\sf Cylr}^{\sf f.w}\Bigl(M \xla{\sf f.w.p.cbl}   \Link_{X_0}(X) \Bigr) 
\longrightarrow
\oC(K)~ ;
\]
taking fiberwise open mapping cylinder over $K$ gives a constructible bundle
\[
{\sf Cylo}^{\sf f.w}\Bigl( \Link_{X_0}(X) \xla{\sf f.w.opn}  X_{|K} \Bigr) 
\longrightarrow
K\times \Delta^1~ .
\]
These constructible bundles base change over $K = K\times \Delta^{\{0\}}$ identically.
Using that $\Bun$ is an $\infty$-category, this pair of concatenating constructible bundles extends to a constructible bundle 
\[
\w{X} \longrightarrow \oC(K\times \Delta^1)\supset \oC(K)\underset{K} \amalg K\times \Delta^1~.
\]
With point (2) above, 
the construction of this constructible bundle supplies the sought data of a lift~(\ref{lift}), sans $\tau$-structures.
We now extend the given fiberwise $\tau$-structure on $X\to \oC(K)$ to one on $\w{X}\to \oC(K\times \Delta^1)$.

Since the tangential structure $\tau$ pulls back along open maps 
(Observation~\ref{stable.ref.loc}), point~(2) just above grants that there is a unique fiberwise $\tau$-structure on the constructible bundle $\Link_{X_0}(X) \to K$ for which the open embedding of point~(2) is one between fiberwise $\tau$-structured constructible bundles.
Furthermore, the fiberwise $\tau$-structure on $X \to \oC(K)$ determines a fiberwise $\tau$-structure on ${\sf Cylr}\bigl(X_0 \xla{\sf f.w.p.cbl} \Link_{X_0}(X)\bigr)  \to \oC(K)$ that restricts to the fiberwise $\tau$-structure on $\Link_{X_0}(X) \to K$.
Composition in the $\infty$-category $\Bun^{\tau}$ endows this extension $\w{X} \to \oC(K\times \Delta^1)$ with a fiberwise $\tau$-structure extending the given one over $\oC(K\times \Delta^{\{1\}})$.
We have exhibited a sought lift~(\ref{lift}).

We lastly show that every factorization is equivalent to that constructed above. 
Let $F\ra \oC(K\times\Delta^1)$ be a constructible bundle, together a requisite fiberwise $\tau$-structure on it and identifications over $\oC(K)$, defining another such lift~(\ref{lift}).
This is classified by a diagram in $\Fun\bigl(\exit(K),\Bun^\tau\bigr)$,
\[
\xymatrix{
&&F_{|K\times\Delta^{\{0\}}}\ar[d]^{\emb}\\
X_0\ar[urr]^-{\sf cls.crt}\ar[rr]_-{\widetilde{X}}&&X_{|K} ,
}
\]
in which $X_0$ is understood
as the constant functor $\exit(K)\ra \Bun^\tau$ valued at $X_0$. 
By taking iterated links, we obtain a diagram in $\Fun\bigl(\exit(K),\Bun\bigr)$:
\[
\xymatrix{
&&\ar[d]_-{(3)}^-{\sf cls.crt}\Link_{X_0}(F_{|K\times\Delta^{\{0\}}})\ar[r]_-{(1)}^-{\sf open}&F_{|K\times\Delta^{\{0\}}}\ar[dd]^-{\emb}\\
X_0 \ar[drr]_-{\sf cls.crt}\ar[urr]^-{\sf cls.crt}&&\ar[d]_-{(2)}^-{\sf open}\Link_{\Link_{X_0}(F_{|K\times\Delta^{\{0\}}})}\bigl(\Link_{X_0}(F_{|K\times\Delta^1})\bigr)
\\
&&
\Link_{X_0}(X_{|K})      \ar[r]_-{(1)}^-{\sf open}
&
X_{|K}    .
}
\]
As argued above, this diagram admits a unique lift to one in $\Fun\bigl(\exit(K),\Bun^\tau\bigr)$, extending the given lift of $c_1 \to \Fun\bigl(\exit(K),\Bun^\tau\bigr)$ classifying $X\to \oC(K)$.
Also as argued above, using that each stratum of $M=X_0$ is both open and closed, the construction of links gives that the horizontal open morphisms labeled as (1) must be embedding morphisms.
Furthermore, the assumption that the morphism $F_{|K\times \Delta^{\{0\}}} \to X_{|K}$ in $\Fun\bigl(\exit(K),\Bun^\tau)$ is by embedding morphisms gives that the top horizontal morphism is in fact an equivalence. 
Since the morphism (2) factors a closed-creation through a closed-creation, it too must be an equivalence. 
This implies the closed-creation morphism (3) must be an equivalence.
We conclude an equivalence in $\Fun\bigl(\exit(K),\Bun^\tau\bigr)$,
\[
\Link_{X_0}(X_{|K})~\simeq~F_{|K\times \Delta^{\{0\}}}~,
\]
under the constant functor at $X_0$ and over the restriction of the given functor classifying $X_{|K}$.

\end{proof}

\begin{remark}
The assumption that $M$ was a smooth manifold, and not a general stratified space, is essential for Lemma~\ref{tau.factor}. 
See Remark \ref{rem.nofact}.
\end{remark}

\subsection{Proof of the tangle hypothesis}

The canonical factorization of Lemma~\ref{lemma.corep} has several consequences.

\begin{cor}\label{cor.actuallyleft}
The functor $\fTang_{n\subset n+k}^{\fr} \ra \Mfd_{n+k}^{\sfr}$ is a left fibration.
\end{cor}
\begin{proof}
For any object $x$ in an $\oo$-category $\cM$, the forgetful functor $\cM^{x/} \ra \cM$ is a left fibration. So, in particular, the functor $\fTang_{n\subset n+k}^{\fr} \simeq (\Mfd_{n+k})^{\RR^k/} \ra \Mfd_{n+k}^{\fr}$ is a left fibration.
\end{proof}

\begin{definition}
The functor 
\[
\Tang_{n\subset n+k}^{\fr} \colon \Mfd_{n+k}^{\sfr}\longrightarrow \Spaces
\]
is the straightening of the left fibration $\fTang_{n\subset n+k}^{\fr} \ra \Mfd_{n+k}^{\sfr}$. 

\end{definition}

\begin{remark}\label{co.rep}
Lemma~\ref{lemma.corep} reveals that the functor $\Tang_{n\subset n+k}^{\fr}\colon \Mfd_{n+k}^{\sfr}\to \Spaces$ is corepresented by $\RR^k$.
\end{remark}

Lemma \ref{lemma.corep} has a further corollary.
Recall the Definition~\ref{def.closed-cover} of a closed sheaf on $\Mfd_n^{\sfr}$.  

\begin{cor}\label{tang.cat}
The functor $\Tang_{n\subset n+k}^{\fr}\colon \Mfd_n^{\sfr}\to \Spaces$ is a closed sheaf.
Furthermore, the composite functor
\[
\Tang_{n\subset n+k}^{\fr}\colon \btheta_{n+k,\emptyset}^{\op} \xra{\lag - \rag} \Mfd_{n+k}^{\vfr} \to \Mfd_{n+k}^{\sfr}\xra{\Tang_{n\subset n+k}^{\fr}} \Spaces
\]
is a pointed flagged $(\infty,n)$-category.  

\end{cor}

\begin{proof}
As pointed out in Remark~\ref{co.rep}, the functor $\Tang_{n\subset n+k}^{\fr}$ is corepresentable.
Therefore, this functor $\Tang_{n\subset n+k}^{\fr}$ carries limit diagrams to limit diagrams.  
Because closed covers in $\Mfd_n^{\sfr}$ are, in particular, limit diagrams, then this functor $\Tang_{n\subset n+k}^{\fr}$ is a closed sheaf.

The functor $\Mfd_n^{\vfr} \to \Mfd_n^{\sfr}$ carries closed cover diagrams to closed cover diagrams.
In particular, the composite functor $\btheta_{n+k,\emptyset}^{\op} \xra{\lag - \rag} \Mfd_{n+k}^{\vfr} \to \Mfd_{n+k}^{\sfr}$ carries (the opposites of) Segal covers to closed covers.
It follows that its restriction to $\bTheta_{n,\emptyset}^{\op}$ carries (the opposites of) closed covers to limit diagrams.  
Lastly, the value $\Tang_{n\subset n+k}^{\fr}(\DD^0) = \{\emptyset\}$ is the single-element set for $k>0$.

\end{proof}

\begin{remark}\label{tang.objects}
For each $0\leq j\leq n+k$, the space of $j$-morphisms of the pointed flagged $(\infty,n+k)$-category $\Tang_{n\subset n+k}^{\fr}$ is the space of framed comdimension-$k$ tangles $W\hookrightarrow  \DD^j$ in the sense of Remark~\ref{rem.clscrtEuc}.
So for $j<k$, the space of $j$-morphisms is contractible, containing only the framed tangle $\emptyset \hookrightarrow \DD^j$.  
As the case that $j=k$, the space of $k$-morphisms is the unordered configuration space
\[
\underset{r\geq 0}\coprod   {\sf Conf}_i(\RR^k)_{\Sigma_k}~\simeq~\Tang_{n\subset n+k}^{\fr}(c_k)
\]
is the space of unordered configurations of points in $\RR^k$.  
\end{remark}

We now turn to the tangle hypothesis, which is a classification of functors from our higher category of tangles $\Tang_{n\subset n+k}^{\fr}$.

\begin{theorem}[Tangle hypothesis]\label{thm.tang}
Assuming Conjecture~\ref{conj.one}, for each pointed $(\infty,n+k)$-category $\uno\in\cC$ with adjoints, evaluation at the 
$k$-endomorphism 
$(\{0\}\subset\RR^k) \in \Tang_{n\subset n+k}^{\fr}$ defines an equivalence
\[
\ev_{\RR^k} \colon \Map^{\ast/}\bigl(\Tang_{n\subset n+k}^{\fr}, \cC\bigr) \xra{~ \simeq ~}  \kEnd_{\cC}(\uno)
\]
between the space of $\cC$-valued pointed functors and the space of $k$-endomorphisms of the distinguished object $\uno\in \cC$.
\end{theorem}

\begin{proof}
This evaluation map factors as a sequence of equivalences among spaces,
\begin{eqnarray}
\nonumber
\Map_{\ast}\bigl(\Tang_{n\subset n+k}^{\fr},\cC\bigr)
&
\underset{\rm Conj~\ref{conj.one.again}}{\simeq}
&
\Map_{\Fun(\Disk_{n+k}^{\sfr},\Spaces)}\Bigl(\Tang_{n\subset n+k}^{\fr},\int \cC\Bigr)
\\
\nonumber
&
\underset{\rm Lem~\ref{lemma.corep}}{\simeq}
&
\Map_{\Fun(\Disk_{n+k}^{\sfr},\Spaces)}\Bigl(\RR^k,\int \cC\Bigr)
\\
\nonumber
&
\underset{\rm Yoneda}{\simeq}
&
\int_{\RR^k}\cC
\\
\nonumber
&
\underset{\rm Thm~\ref{on.Rk}}\simeq
&
\kEnd_{\cC}(\uno)~,
\end{eqnarray}
which we now explain.
The first map is induced by factorization homology.
This map is an equivalence because this factorization homology functor is fully faithful (Corollary~\ref{conj.one.again}).
The second equivalence is implied by the equivalence of Lemma~\ref{lemma.corep}.
The third equivalence is the Yoneda lemma.
The final equivalence is the identification of the factorization homology on Euclidean space (Theorem~\ref{on.Rk}).

\end{proof}

\subsection{Other formulations of the tangle hypothesis}

The identical statement holds for the $(\oo,n+k)$-category formed from the flagged $(\infty,n)$-category $\Tang_{n\subset n+k}^{\fr}$ by forcing the univalence, or completeness, condition to hold. 
The fully faithful functor between presentable $\infty$-categories
\[
\Shv^{\sf unv}(\btheta_n)~\simeq~ \Cat_n~ \hookrightarrow ~\fCat_{n}~ \simeq~ \Shv(\btheta_n)
\]
preserves limits and filtered colimits.
The adjoint functor theorem grants the existence of a left adjoint:
\[
\Shv(\btheta_n)
~ \simeq~ 
\fCat_{n}
\longrightarrow
\Cat_n
~\simeq~ 
\Shv^{\sf unv}(\btheta_n)
~,\qquad
\cC\mapsto \widehat{\cC}~,
\]
which sends a flagged $(\oo,n)$-category to its univalent-completion.

\begin{definition}\label{def.cat.tang}
The pointed $(\oo,n+k)$-category of framed tangles
\[
\widehat{\Tang^{\fr}_{n\subset n+k}}
\]
is the univalent-completion of the pointed flagged $(\oo,n+k)$-category $\Tang_{n\subset n+k}^{\fr}$.
\end{definition}

\begin{cor}\label{cor.univ.tang}
Assuming Conjecture~\ref{conj.one}, for each pointed $(\infty,n+k)$-category $\uno\in\cC$ with adjoints, evaluation at the
$k$-endomorphism 
$(\{0\}\subset\RR^k) \in \widehat{\Tang_{n\subset n+k}^{\fr}}$ defines an equivalence
\[
\ev_{\RR^k}\colon \Map^{\ast/}\bigl(\widehat{\Tang_{n\subset n+k}^{\fr}}, \cC\bigr) \ \simeq \  \kEnd_{\cC}(\uno)
\]
between the space of $\cC$-valued pointed functors and the space of $k$-endomorphisms of the distinguished object $\uno\in \cC$.
\end{cor}
\begin{proof}
Univalent-completion determines the top horizontal arrow in the diagram among spaces:
\[
\xymatrix{
\Map^{\ast/}\bigl(\Tang_{n\subset n+k}^{\fr}, \cC\bigr)  \ar[rr]^-{\widehat{(-)}}  \ar[dr]_-{\ev_{\RR^k}}
&&
\Map^{\ast/}\bigl(\widehat{\Tang_{n\subset n+k}^{\fr}}, \widehat{\cC}\bigr)   \ar[dl]^-{\ev_{\RR^k}}
\\
&
\kEnd_{\cC}(\uno)
&
.
}
\]
By assumption, $\cC$ is univalent-complete: $\cC\simeq \widehat{\cC}$.
From the universal property of the univalent-completion as a left adjoint, the top horizontal map is an equivalence.
Theorem~\ref{thm.tang} states that the down-rightward map is an equivalence.  
It follows that the down-leftward map is an equivalence, as desired.

\end{proof}

The tangle hypothesis has yet another form in terms of $\cE_k$-monoidal higher categories. 
There is an adjunction between $\infty$-categories
\begin{equation}\label{B.Omega}
\fB^k\colon \Alg_{\cE_k}(\Cat_n)~\rightleftarrows~ \Cat_{n+k}^{\ast/}\colon \Omega^k~.
\end{equation}
The left adjoint is fully faithful; its image consists of those pointed $(\infty,n+k)$-categories $\ast \to \cC$ for which the functor $\ast \xra{\simeq} \cC_{<k}$ to the maximal $(\infty,k-1)$-subcategory is an equivalence.

\begin{remark}
The colocalization~(\ref{B.Omega}) can be obtained as follows.  
Let $\cV$ be a symmetric monoidal $\infty$-category.  
In~\S6.3 of~\cite{gepnerhaugseng}, the authors construct a colocalization among $\infty$-categories
\[
\fB^k\colon \Alg_{\cE_k}(\cV)~\rightleftarrows~ \Cat_{k}(\cV)^{1_\cV/}\colon \Omega^k
\]
involving that of $\cE_k$-algebras in $\cV$ and that of $(\infty,k)$-categories enriched over $\cV$ equipped with a morphism from the $1_\cV$, the $\cV$-enrichment of the terminal $(\infty,k-1)$-category $\ast$ by the symmetric monoidal unit $1_\cV\in \cV$.
The colocalization~(\ref{B.Omega}) follows by way of the equivalence 
\[
\Cat_{n+k}~\simeq~ \Cat_k(\Cat_n)
\]
of Remark~\ref{rem.enriched}.
The identification of the image of $\fB^k$ is immediate from the construction of this colocalization.

\end{remark}

\begin{remark}
Let $\uno\in \cC$ be a pointed $(\infty,n+k)$-category.
Consider the underlying $(\infty,n)$-category $\Omega^k\cC$.  
Its factorization homology evaluates as
\[
\Omega^k\cC\colon \Mfd_{n}^{\vfr} \ni D\mapsto \int_{D} \Omega^k\cC~\simeq ~\int_{\RR^k\times D} \cC~\in \Spaces~,
\]
which is in terms of the factorization homology of $\cC$.
The $\cE_k$-monoidal structure on the $(\infty,n)$-category $\Omega^k\cC$ can be presented as follows.
Consider the $\infty$-subcategory $\Disk_k^{\fr}\subset \Mfd_k^{\vfr}$ consisting of those objects that are finite disjoint unions of $\RR^k$ and morphisms which are composites of open embeddings and closed morphisms.
Taking connected components defines a functor $\Disk_k^{\fr} \to \Fin_\ast$ to the category of based finite sets; arranged in this way, $\Disk_k^{\fr}$ is equivalent to the $\infty$-operad $\cE_k$.
Precompose the factorization homology of $\cC$ with the product functor:
\[
\Disk_k^{\fr}\times \cMfd_n^{\vfr} \xra{~\times~} \Mfd_{n+k}^{\vfr} \xra{~\displaystyle \int \cC~}\Spaces~.
\]
The $\cE_k$-monoidal structure on $\Omega^k\cC$ is the adjoint to this composite functor:
\[
\Omega^k\cC\colon \Disk_k^{\fr} \longrightarrow \Fun(\cMfd_n^{\vfr},\spaces)
~,\qquad
(\RR^k)^{\sqcup I}~\mapsto \int_{(\RR^k)^{\sqcup I} \times -}\cC \xra{\simeq} \Bigl(\int_{\RR^k\times -}\cC\Bigr)^{\times I}~.
\]
\end{remark}

\begin{definition}[Duals]\label{def.duals}
For $k>0$, an $\cE_k$-monoidal $(\infty,n)$-category $\cC$ has \emph{duals} if each $k$-morphism in the $(\infty,n+k)$-category $\fB^k \cC$ has a left and a right adjoint in the sense of Definition~\ref{def.adjoints}.

\end{definition}

\begin{definition}\label{def.mon.tang}
The \emph{$\cE_k$-monoidal $(\oo, n)$-category of framed tangles}
\[
\Omega^k \widehat{\Tang^{\fr}_{n\subset n+k}}
\]
is that associated to the pointed $(\oo,n+k)$-category $\widehat{\Tang^{\fr}_{n\subset n+k}}$ via the functor $\Omega^k\colon  \Cat_{n+k}^{\ast/} \longrightarrow \Alg_{\cE_k}(\Cat_n)$.

\end{definition}

\begin{remark}\label{tang.obj}
We follow up on Remark~\ref{tang.objects}.
For $0\leq j\leq n$, a $j$-morphism in the underlying $(\infty,n)$-category of the $\cE_k$-monoidal $(\infty,n)$-category $\Omega^k \widehat{\Tang^{\fr}_{n\subset n+k}}$ is represented by a framed codimension-$k$ tangle $W\hookrightarrow \RR^k\times \DD^j$ in the sense of Remark~\ref{rem.clscrtEuc}.

\end{remark}

The next result states that the $\cE_k$-monoidal $(\infty,n)$-category of framed tangles of Definition~\ref{def.mon.tang} is equivalent data as the pointed $(\infty,n+k)$-category of framed tangles of Definition~\ref{def.cat.tang}.
\begin{prop}\label{tang.is.reduced}
The pointed $(\infty,n+k)$-category $\widehat{\Tang_{n\subset n+k}^{\fr}}$ is in the image of the functor $\fB^k$.
In other words, the counit of the adjunction~(\ref{B.Omega}),
\[
\fB^k\Omega^k \widehat{\Tang_{n\subset n+k}^{\fr}} \xra{~\simeq~} \widehat{\Tang_{n\subset n+k}^{\fr}}~,
\]
is an equivalence.  
\end{prop}  
\begin{proof}
We must show, for each $0\leq i <k$, that the map classifying the empty tangle, $\ast \xra{\{\emptyset\}}\widehat{\Tang_{n\subset n+k}^{\fr}}(\DD^i)$, is an equivalence between spaces.
Through Lemma~\ref{lemma.corep}, this map is identified as the map $\ast \xra{0}  \Mfd_{n+k}^{\sfr}(\RR^k,\DD^i)$ selecting the composite morphism $\RR^k \xra{\sf cls} \emptyset \xra{\emb}\DD^i$.
In light of the closed-active factorization system on $\Mfd_{n+k}^{\sfr}$, it is enough to argue that there are no active morphisms from $\RR^k$ to $\DD^i$.  
But the target of any such active morphism must have dimension greater than that of $\RR^k$.
The assumption that $i<k$ thereby ensures that no such active morphisms exists, as desired.

\end{proof}

\begin{cor}\label{cor.ek.tang}
Assuming Conjecture \ref{conj.one}, 
for each $\cE_k$-monoidal $(\oo,n)$-category $\fX$ with duals and adjoints, evaluation at the object $(\{0\}\subset\RR^k) \in \Omega^k \Tang_{n\subset n+k}^{\fr}$ defines an equivalence
\[
\ev_\ast\colon \Map_{\Alg_{\cE_k}(\Cat_n)}\Bigl(\Omega^k\widehat{\Tang_{n\subset n+k}^{\fr}}, \fX\Bigr) \ \simeq \  \obj(\fX)
\]
between the space of $\fX$-valued $\cE_k$-monoidal functors and the space of objects of $\fX$.
\end{cor}
\begin{proof}
The deloop functor $\fB^k$ determines the top horizontal arrow in the diagram among spaces:
\[
\xymatrix{
\Map^{\ast/}\bigl(\Omega^k\Tang_{n\subset n+k}^{\fr}, \fX\bigr)  \ar[rr]^-{\fB^k}  \ar[dr]_-{\ev_\ast}
&&
\Map^{\ast/}\bigl(\fB^k\Omega^k \widehat{\Tang_{n\subset n+k}^{\fr}}, \fB^k\fX \bigr)   \ar[dl]^-{\ev_{\RR^k}}
\\
&
\kEnd_{\cC}(\uno)
&
.
}
\]
Because $\fB^k$ is fully faithful, the top horizontal map is an equivalence.  
Proposition~\ref{tang.is.reduced} gives an identification $\fB^k \Omega^k\widehat{\Tang_{n\subset n+k}^{\fr}} \xra{\simeq} \widehat{\Tang_{n\subset n+k}^{\fr}}$ between pointed $(\infty,n+k)$-categories.  
Through this identification, the down-leftward map is identified as the equivalence in Corollary~\ref{cor.univ.tang}.
It follows that the down-rightward map is an equivalence, as desired.

\end{proof}

\section{The cobordism hypothesis}

A principal result from differential topology is that the space of embeddings of a smooth manifold $M$ into an infinite-dimensional Euclidean space $\Emb(M,\RR^{\oo})$ is contractible.
Consequently, for many purposes manifolds may as well be regarded as subspaces of an ambient Euclidean space $\RR^{\oo}$.
There is a further equivalence
\[
\Emb(M,\RR^{\oo}) ~\simeq~ \underset{k\geq 0}\colim \Emb(M,\RR^k)
\]
with the sequential colimit of the spaces of embeddings into finite-dimensional Euclidean spaces. In particular, any two embeddings of $M$ into $\RR^{k}$ are isotopic through embeddings in a sufficiently large Euclidean space $\RR^{k+N}$. This offers an avenue for achieving results about abstract manifolds $M$ from results about embedded submanifolds $M\subset\RR^k$, by using the tubular neighborhood of $M$ in $\RR^k$. Examples include Hopf's proof of the Poincar\'e--Hopf theorem (see \cite{milnor.top}), and pre-Chern approaches to the Gauss--Bonnet theorem, from Gauss to Allendoerfer--Weil \cite{allendoerfer.weil}, who made use of Weyl's tube formula \cite{weyl}. Our proof of the cobordism hypothesis from the tangle hypothesis shares this character.

Note the functor 
\[
\RR\times - \colon \Mfd_{n+k}^{\sfr} \longrightarrow \Mfd_{n+k+1}^{\sfr}
~,\qquad
M\mapsto \RR\times M~.
\]
This functor carries $\RR^k$ to $\RR^{k+1}$.
Through the identification $\Tang_{n\subset n+k}^{\fr}\simeq \RR^k$ from Lemma~\ref{lemma.corep}, there results a functor 
\begin{equation}\label{6}
\Omega^k\Tang_{n\subset n+k}^{\fr} \longrightarrow \Omega^{k+1}\Tang_{n\subset n+k+1}^{\fr}
\end{equation}
between $(\infty,n)$-categories.  
In terms of framed tangles, for each object $D\in \Mfd_{n+k}^{\sfr}$, this map~(\ref{6}) sends a codimension-$k$ tangle $\RR^k\times W\subset D$ to the codimension-$(k+1)$ tangle $\RR^{k+1}\times W\subset \RR\times W$. 
The functors~(\ref{6}) are evidently functorial in the poset element $k\in \NN$.  
The contractibility of $\Emb(M,\RR^{\oo})$ justifies the following.
\begin{definition}\label{def.bord}
The \emph{flagged $(\oo,n)$-category of $n$-framed cobordisms} is the sequential colimit
\[
\Bord_n^{\fr}:= \underset{k\geq 0}\colim \ \Omega^k \Tang_{n\subset n+k}^{\fr}
\]
of the $k$-fold loops of the pointed flagged $(\oo,n+k)$-categories of codimension-$k$ framed tangles.
\end{definition}

The next sections equip $\Bord_n^{\fr}$ with a symmetric monoidal structure.

\subsection{Delooping}\label{sec.deloop}
We present a formalism for constructing symmetric monoidal $(\infty,n)$-categories from sequences of compatible $\cE_k$-monoidal $(\infty,n)$-categories.
This is much like constructing an $\infty$-loop space from a pre-spectrum.

For each dimension $k$ consider the symmetric monoidal $\infty$-category $\Disk_k^{\fr}$ of finite disjoint unions of framed $k$-dimensional vector spaces and framed open embeddings among them, with symmetric monoidal structure given by disjoint union.
Taking products with Euclidean spaces assembles these symmetric monoidal $\infty$-categories into a functor $\Disk_\bullet^{\fr}\colon \NN \to \Alg_{\sf Com}(\Cat)$.
The colimit of this functor is canonically identified as the symmetric monoidal envelope of the commutative $\infty$-operad.  
Therefore, for each symmetric monoidal $\infty$-category $\cV$, applying $\Fun^\ot(-,\cV)$ defines a functor
$
\Alg_{\cE_-}(\cV)\colon \NN^{\op} \to \Cat
$
whose limit is $\Alg_{\sf Com}(\cV)$.  
Consider the Cartesian fibration
\[
\Alg_{\cE_\bullet}(\cV) \longrightarrow \NN
\]
which is the unstraightening of this functor.
There is a fully faithful functor to the $\infty$-category of sections of this Cartesian fibration
\begin{equation}\label{as-Carts}
\Alg_{\sf Com}(\cV) \longrightarrow \Gamma\Bigl(\Alg_{\cE_\bullet}(\cV) \longrightarrow \NN\Bigr)
\end{equation}
whose image consists of those sections that carry morphisms to Cartesian morphisms.  
Explicitly, an object in the righthand $\infty$-category is the data of an $\cE_k$-algebra $A_k$ for each $k\geq 0$ together with a map of $\cE_k$-algebras $A_k \to (A_{k'})_{|\cE_k}$ for each $k\leq k'$, coherently compatibly; while an object in the image of~(\ref{as-Carts}) is one for which each $A_k\to (A_{k'})_{|\cE_k}$ is an equivalence.

\begin{lemma}\label{spectra}
Let $\cV$ be a Cartesian presentable $\infty$-category, which we regard as a symmetric monoidal $\infty$-category via the Cartesian product.
The functor~(\ref{as-Carts}) admits a right adjoint that implements a colocalization.
This right adjoint evaluates on a section $(k\mapsto A_k)$ as the Cartesian section
\[
k~\mapsto~ \underset{\ell \geq k} \colim (A_\ell)_{|k}~: = ~\colim \Bigl( A_k \to (A_{k+1})_{|\cE_{k}} \to (A_{k+2})_{|\cE_k} \to  \dots \Bigr)~.  
\]
\end{lemma}

\begin{proof}
The colimit of the functor $\NN^{-/}\colon \NN^{\op} \to \Cat$, given by $k\mapsto \NN_{\geq k}$, is canonically identified as $\NN$.
Therefore, to construct the alleged right adjoint it is enough to construct a downward morphism between functors $\NN^{\op} \to \Cat_{\infty}$
\[
\Small
\xymatrix{
\dots  \ar[r]
&
\Gamma\Bigl(\Alg_{\cE_{\bullet> k}}(\cV) \longrightarrow \NN_{>k}\Bigr)  \ar[d]  \ar[r]
&
\Gamma\Bigl(\Alg_{\cE_{\bullet\geq k}}(\cV) \longrightarrow \NN_{\geq k}\Bigr)  \ar[d]  \ar[r]
&
\dots
\\
\dots  \ar[r]
&
\Alg_{\cE_{k+1}}(\cV)  \ar[r]
&
\Alg_{\cE_k}(\cV)   \ar[r]
&
\dots .
}
\]
and then argue its properties.
Because, for each $k\in \NN$ the object $(k=k)$ is final in $(\NN^{k/})^{\op}$, the colimit of the composite functor $\NN_{\geq k}^{\op} \to \NN^{\op} \xra{\Alg_{\cE_-}(\cV)} \Cat$
is canonically identified as $\Alg_{\cE_k}(\cV)$.  
There results a functor between $\infty$-categories
$\Gamma\Bigl(\Alg_{\cE_{\bullet\geq k}}(\cV) \longrightarrow \NN_{\geq k}\Bigr)
\xra{(-)_{|\cE_k}}
\Fun\bigl(\NN_{\geq k}^{\op}, \Alg_{\cE_k}(\cV)\bigr).
$  
Postcomposing this functor with the colimit functor $\Fun\bigl(\NN_{\geq k}^{\op}, \Alg_{\cE_k}(\cV)\bigr) \xra{\colim} \Alg_{\cE_k}(\cV)$ defines the functor we seek for each given $k\in \NN$.  
Since the forgetful functor $\Alg_{\cE_k}(\cV) \to \cV$ preserves and creates filtered colimits, this colimit functor indeed exists.  
This also implies that each square in the diagram displayed above canonically commutes.  
We conclude a functor between limit $\infty$-categories 
$\Gamma\Bigl(\Alg_{\cE_\bullet}(\cV ) \longrightarrow \NN\Bigr)
\to
\Alg_{\sf Com}(\cV)$.

The functor $(-)_{|\cE_k}$ above carries Cartesian sections to constant functors.  
Because $\NN$ has contractible classifying space, the composite functor
$\Alg_{\sf Com}(\cV) \to \Gamma\Bigl(\Alg_{\cE_\bullet}(\cV ) \longrightarrow \NN\Bigr) \to \Alg_{\sf Com}(\cV)$ is canonically identified as the identity functor.  
Constructed by way of a colimit, there is a unit transformation from the composite functor
$\Gamma\Bigl(\Alg_{\cE_\bullet}(\cV ) \longrightarrow \NN\Bigr) \to \Alg_{\sf Com}(\cV) \to \Gamma\Bigl(\Alg_{\cE_\bullet}(\cV ) \longrightarrow \NN\Bigr)$ 
to the identity functor.
Because $\NN$ is filtered, the restriction of this unit transformation to the Cartesian sections is a natural equivalence.  
This completes the proof of the lemma.

\end{proof}

\subsection{Construction of $\Bord_n^{\fr}$}

We define the symmetric monoidal flagged $(\infty,n)$-category of $n$-framed cobordisms 
as the colimit of the $\cE_k$-monoidal $(\infty,n)$-categories of framed codimension-$k$-tangles.

Consider the full $\infty$-subcategory $\Mfld_k^{\fr}\subset \Mfd_k^{\vfr, \sf emb}$ consisting of those vari-framed $n$-manifolds each of whose strata has dimension precisely $k$.  
Taking products defines a functor between $\infty$-categories
$
\Mfd_n^{\vfr} \times \Mfld_k^{\vfr}  \to \Mfd_{n+k}^{\vfr}.
$
By inspection, this functor carries disjoint unions in the second coordinate to closed covers, and it carries closed covers in the second coordinate to closed covers.  
As so, restriction of this functor along 
$
\btheta_n^{\op}\times\Disk_n^{\fr} \xra{~-\times\lag - \rag~} \cMfd_n^{\vfr} \times \Mfld_k^{\fr}
$
defines a functor
\begin{equation}\label{deloop}
\Omega^k\colon \fCat_{n+k}^{\ast/} \longrightarrow   \Alg_{\cE_k}\bigl({\fCat_n}\bigr)~,\qquad \cC\mapsto k\cE{\sf nd}_\cC(\uno)
\end{equation}
from pointed $(\infty,n+k)$-categories to $\cE_k$-monoidal flagged $(\infty,n)$-categories.  
As indicated, the value of this functor on a pointed $(\infty,n+k)$-category $\ast \xra{\uno} \cC$ is the $(\infty,n)$-category of $k$-endomorphisms of its point, as it inherits a $\cE_k$-monoidal structure.

The framed tangle $\infty$-categories of Definition~\ref{def.fr.tang} assemble as a functor
\[
\fTang_{n\subset n+\bullet}^{\fr}\colon \NN \longrightarrow \Cat
\]
over the functor $\Mfd_{n+\bullet}^{\sfr}$.  
There results a section
\begin{equation}\label{tang-sec}
\NN \longrightarrow \Alg_{\cE_\bullet}({\fCat_n})~,\qquad k\mapsto \Omega^k\Tang_{n\subset n+k}^{\fr}~.
\end{equation}

\begin{definition}\label{def.Bord-B}
The symmetric monoidal flagged $(\infty,n)$-category of $n$-framed cobordisms
\[
\Bord_n^{\fr}~:=~\underset{k\geq 0} \colim ~\Omega^k\Tang_{n\subset n+k}^{\fr}
\]
is the value of the right adjoint functor of Lemma~\ref{spectra} on the section~(\ref{tang-sec}).  

\end{definition}

\begin{remark}
We follow up on Remark~\ref{tang.objects}.  
Let $0\leq j \leq n$.
From its definition, the space of $j$-morphisms in the underlying flagged $(\infty,n)$-category of the symmetric monoidal flagged $(\infty,n)$-category $\Bord_n^{\fr}$ is identified as a moduli space:
\begin{eqnarray}
\nonumber
\Bord_n^{\fr}(c_j)
&
:=
&
\underset{k\geq 0} \colim~ \Omega^k \Tang_{n\subset n+k}^{\fr}(c_j)
\\
\nonumber
&
:=
&
\underset{k\geq 0} \colim ~\Map^{\ast/}\bigl(\fC(\RR^k\times \DD^j),\Tang_{n\subset n+k}^{\fr}\bigr)
\\
\nonumber
&
:=
&
\underset{k\geq 0} \colim  ~(\fTang_{n\subset n+k}^{\fr})_{|\RR^k\times \DD^j}
\\
\nonumber
&
\underset{\rm Rmk~\ref{rem.clscrtEuc}}{~\simeq~}
&
\underset{k\geq 0} \colim ~   \Bigl| \Bigl\{W\overset{\rm framed}{\underset{{\rm codim\text{-}}k}\hookrightarrow} \RR^k\times \DD^j  \Bigr\} \Bigr|
\\
\nonumber
&
\underset{\rm project}{\xra{~\simeq~}}
&
\Bigr|\Bigl\{ W^j \to \DD^j~,~ \varphi  \Bigr\}\Bigl|~.
\end{eqnarray}
Here, the final term is a moduli space of the following data:
\begin{itemize}
\item a compact smooth $j$-manifold $W$ with corners;

\item a smooth map $W\to \DD^j$ with respect to which the corner structure of $W$ is pulled back from that of the hemispherical disk $\DD^j$;

\item an injection $\varphi\colon \sT_W \hookrightarrow \epsilon^j_W$ of the constructible tangent bundle of $W$ into the trivial rank $j$ vector bundle over $W$.

\end{itemize}

\end{remark}

\subsection{Proof of the cobordism hypothesis}
We now prove our main result.
Also, we use the notation $\Map^\ot(-,-)$ for spaces of morphisms in the $\infty$-category $\Alg_{\sf Com}({\Cat_n})$ of symmetric monoidal $(\infty,n)$-categories.
\begin{theorem}[Cobordism hypothesis]\label{B-bord}
Let $\fX$ be a symmetric monoidal $(\infty,n)$-category with adjoints and duals.
Evaluation at $\ast\in \Bord_n^{\fr}$ determines an equivalence between space
\[
\ev_\ast\colon \Map^{\ot}\bigl(\Bord_n^{\fr},\fX\bigr)\xra{~\simeq~}\obj(\fX)
\]
from $\fX$-valued symmetric monoidal functors to the space of objects of $\fX$.

\end{theorem}

\begin{proof}
We use the definition of $\Bord_n^{\fr}$ as a colimit of tangle categories, and then apply the form of the tangle hypothesis given by Corollary \ref{cor.ek.tang}:
\begin{eqnarray}
\nonumber
\Map^\ot\bigl(\Bord_{n}^{\fr},\fX\bigr)
&
\simeq
&
\underset{k\to \infty}\limit\Map_{\Alg_{\cE_k}(\fCat_n)}\Bigl(\Omega^k\Tang_{n\subset n+k}^{\fr}, \fX\Bigr)
\\
\nonumber
&
\simeq
&
\underset{k\to \infty}\limit \obj(\fX)
\\
\nonumber
&
\simeq
&
\obj(\fX)~.
\end{eqnarray}
The result follows, since we obtain a constant sequential limit with value $\obj(\fX)$.

\end{proof}

\begin{definition}\label{def.cat.bord}
The symmetric monoidal $(\oo,n)$-category of framed cobordisms
\[
\widehat{\Bord^{\fr}_n}
\]
is the univalent-completion of the symmetric monoidal $(\oo,n)$-category $\Bord_n^{\fr}$.
\end{definition}

\begin{cor}\label{cor.univ.bord}
Assuming Conjecture~\ref{conj.one}, for each symmetric monoidal $(\infty,n)$-category $\fX$ with adjoints and duals, evaluation at the 
$\ast \in \widehat{\Bord_n^{\fr}}$ defines an equivalence
\[
\ev_\ast\colon \Map^{\ot}\Bigl(\widehat{\Bord_n^{\fr}},\fX\Bigr)\xra{~\simeq~}\obj(\fX)
\]
between the space of $\fX$-valued symmetric monoidal functors and the space of objects of $\fX$.
\end{cor}

\subsection{Invertible field theories}
We recover the expected classification of \emph{invertible} framed topological quantum field theories.
This is phrased in terms of the groupoid completion of the framed cobordism category.

The fully faithful inclusion $\Spaces^{\ast/} \hookrightarrow \Cat_{n+k}^{\ast/}$ as pointed $\infty$-groupoids has a left adjoint localization:
\begin{equation}\label{B.def}
\sB \colon \Cat_{n+k}^{\ast/} ~\rightleftarrows~ \Spaces^{\ast/}~;
\end{equation}
this left adjoint carries each pointed $(\infty,n+k)$-category to its \emph{$\infty$-groupoid-completion}.
Note the commutative diagram among $\infty$-categories:
\begin{equation}\label{endo.loops}
\xymatrix{
\Spaces^{\ast/} \ar[rr]  \ar[dr]_-{\Omega^k}
&
&
\Cat_{n+k}^{\ast/}  \ar[dl]^-{\kEnd(\uno)}
\\
&
\Spaces
&
.
}
\end{equation}

\begin{cor}\label{B.Tang}
There is a canonical identification between pointed spaces
\[
\sB\Bigl(\widehat{\Tang_{n\subset n+k}^{\fr}}\Bigr)~\simeq~ S^{k}
\]
involving the $\infty$-groupoid-completion of the pointed $(\infty,n+k)$-category of framed tangles and the pointed $k$-sphere.

\end{cor}

\begin{proof}
Let $\ast \xra{\uno} X$ be a pointed space.
Regard $X$ as a pointed $(\infty,n+k)$-category in which each $i$-morphism is invertible.
We explain the following sequence of canonical identifications
\begin{eqnarray}
\nonumber
\Map^{\ast/}\Bigl(\sB \bigl(\widehat{\Tang_{n\subset n+k}^{\fr}}\bigr), X\Bigr)
&
\simeq
&
\Map^{\ast/}(\widehat{\Tang_{n\subset n+k}^{\fr}}, X)
\\
\nonumber
&
\simeq
&
\kEnd_X(\uno)
\\
\nonumber
&
\simeq
&
\Omega^{k}X.
\end{eqnarray}
The first identification is the localization~(\ref{B.def}).
The second identification is Theorem~\ref{B-bord}.
The third identification is the commutativity of the diagram~(\ref{endo.loops}).
In this way, we see that the functors $\Spaces^{\ast/}\to \Spaces$ corepresented by the $\infty$-groupoid-completion of $\widehat{\Tang_{n\subset n+k}^{\fr}}$ and the pointed sphere $S^{k}$ are canonically identified.
The desired canonical equivalence between pointed spaces then follows from the Yoneda lemma.

\end{proof}

Consider the fully faithful inclusion $\Spectra_{\geq 0}\hookrightarrow \Alg_{\sf Com}(\spaces) \hookrightarrow \Alg_{\sf Com}(\Cat_n)$ of connective spectra as symmetric monoidal $\infty$-groupoids with duals.
Note the commutative diagram among $\infty$-categories:
\begin{equation}\label{obj.loops}
\xymatrix{
\Spectra_{\geq 0}  \ar[r]  \ar[dr]_-{\Omega^\infty}
&
\Alg_{\sf Com}(\Spaces)  \ar[r]  \ar[d]^-{\rm forget}
&
\Alg_{\sf Com}(\Cat_n)  \ar[dl]^-{\obj}
\\
&
\Spaces
&
.
}
\end{equation}
This fully faithful functor is a right adjoint in a localization:
\begin{equation}\label{B.ot}
\sB^{\otimes}\colon \Alg_{\sf Com}(\Cat_{n}) ~\rightleftarrows~ \Spectra_{\geq 0}~;
\end{equation}
this left adjoint carries each symmetric monoidal $(\infty,n)$-category to its \emph{Picard $\infty$-groupoid-completion}.

\begin{cor}\label{B.Bord}
There is a canonical identification between spectra
\[
\sB^{\otimes} \widehat{\Bord_n^{\fr}}
~ \simeq ~
\SS
\]
involving the Picard $\infty$-groupoid-completion of the symmetric monoidal $(\infty,n)$-category of framed cobordisms and the sphere spectrum.

\end{cor}

\begin{proof}
Let $X$ be a connective spectrum.
Regard $X$ as a symmetric monoidal $(\infty,n)$-category in which each $i$-morphism is invertible and each object has a dual.
We explain the following sequence of canonical identifications
\begin{eqnarray}
\nonumber
\Map_{\Spectra}(\sB^\ot \widehat{\Bord_n^{\fr}}, X)
&
\simeq
&
\Map^\ot(\widehat{\Bord_n^{\fr}}, X)
\\
\nonumber
&
\simeq
&
\obj(X)
\\
\nonumber
&
\simeq
&
\Omega^{\oo}X.
\end{eqnarray}
The first identification is the localization~(\ref{B.ot}).
The second identification is Theorem~\ref{B-bord}.
The third identification is the commutativity of the diagram~(\ref{obj.loops}).
In this way, we see that the functors $\Spectra_{\geq 0}\to \Spaces$ corepresented by the Picard $\infty$-groupoid completion $\sB^{\ot} \widehat{\Bord_n^{\fr}}$ and the sphere spectrum $\SS$ are canonically identified.
The desired canonical equivalence between connective spectra follows from the Yoneda lemma.

\end{proof}

Corollary~\ref{B.Bord} has the following immediate consequence.
\begin{cor}\label{invertible.theories}
Let $\fX$ be a symmetric monoidal $(\infty,n)$-category.
Let $\fX^\sim\subset \fX$ be its maximal symmetric monoidal $\infty$-groupoid with duals; regard it as a connective spectrum.
There is a monomorphism between pointed spaces
\[
\Omega^\infty \fX^\sim~\hookrightarrow~ \Map^{\ot}\bigl(\Bord_n^{\fr},\fX\bigr)~;
\]
the image consists of those functors that carry each $i$-morphism in $\Bord_n^{\fr}$ to an invertible $i$-morphism in $\fX$.

\end{cor}

\section{Enrichments}
In favorable situations, self-enrichment is adjoint to a monoidal structure.
In our setting, we have two distinct monoidal structures.
One is the Cartesian product of higher categories;
another comes from geometric product of stratified spaces. 
The higher categories $\fC(\DD^i \times \DD^j)$ and $c_i\times c_j$ are not equivalent, and thus we have two distinct self-enriched enhancements of the cobordism hypothesis.\footnote{The higher category $\fC(\DD^i\times\DD^j)$ appears to be a form of the Gray tensor product of $c_i$ and $c_j$.}

\subsection{The Cartesian enrichment of the tangle hypothesis}\label{cat.n.enrich}
The presentable $\infty$-category $\Cat_{n+k}$ of $(\infty,n+k)$-categories has the property that the product functor
\[
\times \colon \Cat_{n+k}\times \Cat_{n+k} \longrightarrow \Cat_{n+k}~,\qquad (\cC,\cD)\mapsto \cC\times \cD~,
\]
preserves colimits in each variable (see~\cite{rezk-n}).
It follows that $\Cat_{n+k}$ is naturally enriched over $\Cat_{n+k}$, as in~\cite{gepnerhaugseng}.
Namely, for $\cC,\cD\in \Cat_{n+k}$, the $(\infty,n+k)$-category $\Fun(\cC,\cD)$ of morphisms from $\cC$ to $\cD$ represents the presheaf
\[
\Map\bigl(-,\Fun(\cC,\cD)\bigr)\colon \Cat_{n+k}^{\op} \ni \cK \mapsto \Map(\cK\times \cC,\cD)\in \Spaces~.
\]

Likewise, smash product between pointed $(\infty,n+k)$-categoreis defines a functor
\[
\smsh \colon \Cat_{n+k}^{\ast/}\times \Cat_{n+k}^{\ast/} \longrightarrow \Cat_{n+k}^{\ast/}~,\qquad (\cC,\cD)\mapsto \cC\smsh \cD:= \ast \underset{\cC \underset{\ast} \amalg \cD}\amalg \cC\times \cD~.
\]
Using that products distribute over colimits in each variable, so does this smash product functor.  
Using presentability of the $\infty$-category $\Cat_{n+k}^{\ast/}$ is naturally enriched over $\Cat_{n+k}^{\ast/}$.  
Namely, for $\cC,\cD\in \Cat_{n+k}^{\ast/}$, the pointed $(\infty,n+k)$-category $\Fun_\ast(\cC,\cD)$ of morphisms from $\cC$ to $\cD$ represents the presheaf
\begin{equation}\label{19}
\Map^{\ast/}\bigl(-,\Fun_\ast(\cC,\cD)\bigr)\colon (\Cat_{n+k}^{\ast/})^{\op} \ni \cK \mapsto \Map^{\ast/}(\cK\smsh \cC,\cD)\in \Spaces~.
\end{equation}

In particular, there is a commutative diagram of $\oo$-categories
\begin{equation}\label{20}
\xymatrix{
&\Cat_{n+k}^{\ast/} \ar[dr]^-{\sf obj}\\
(\Cat_{n+k}^{\ast/})^{\op}\times \Cat_{n+k}^{\ast/} 
\ar[ur]^-{\Fun_\ast}\ar[rr]_-{\Map^{\ast/}}&& \spaces.
}
\end{equation}

We prove the next result at the end of this subsection.

\begin{theorem}\label{thm.cart.tang}
Assuming Conjecture~\ref{conj.one}, for each pointed $(\infty,n+k)$-category $\uno\in\cC$ with adjoints, evaluation at the object $(\{0\}\subset\RR^k) \in \Tang_{n\subset n+k}^{\fr}$ defines an equivalence
\[
\ev_{\RR^k}\colon \Fun_\ast\bigl(\Tang_{n\subset n+k}^{\fr}, \cC\bigr) \ \simeq \  \kEnd_{\cC}(\uno)
\]
between the pointed $(\oo,n+k)$-category of $\cC$-valued pointed functors and the pointed space of $k$-endomorphisms of the distinguished object $\uno\in \cC$. 
In particular, for each $i>0$, each $i$-morphism in $\Fun_\ast\bigl(\Tang_{n\subset n+k}^{\fr}, \cC\bigr)$ is invertible.
\end{theorem}

Taking coproduct with the terminal $(\infty,n+k)$-category defines a functor between $\infty$-categories,
\[
(-)_\ast \colon \Cat_{n+k} \longrightarrow \Cat_{n+k}^{\ast/}~,\qquad \cT\mapsto \cT_\ast := (\ast \to  \cT\amalg \ast)~,
\]
from that of $(\infty,n+k)$-categories to that of pointed $(\infty,n+k)$-categories.
Observe, for each pointed $(\infty,n+k)$-category $\cC$, and each $(\infty,n+k)$-category $\cT$, the canonical functor 
\[
\cC \smsh \cT_\ast \longrightarrow \cC
\]
between pointed $(\infty,n+k)$-categories.

\begin{lemma}\label{lem.smshcell}
Let $k>0$.
Let $\cT$ be an $(\infty,n+k)$-category.
Provided the $\infty$-groupoid completion $\sB^n \cT \simeq \ast$ is contractible, the canonical functor between pointed $(\infty,n+k)$-categories
\[
c_{k/\partial c_k}\smsh \cT_\ast \xra{~\simeq~}c_{k/\partial c_k}
\]
is an equivalence.
\end{lemma}
\begin{proof}
Being defined as a colimit, smash products distribute over colimits.
Therefore, the smash product $c_{k/\partial c_k}\smsh \cT_\ast$ is identified as the pushout among pointed $(\oo,n+k)$-categories
\[
\xymatrix{
\partial c_{k,\ast}\smsh \cT_\ast \ar[d]\ar[r]&c_{k, \ast}\smsh \cT_\ast \ar[d]\\
\ast\smsh \cT_\ast\ar[r]&c_{k/\partial c_k}\smsh \cT_\ast    .
}   
\]
This diagram is identified as the pushout diagram
\[
\xymatrix{
(\partial c_k\times \cT)_\ast \ar[d]\ar[r]&(c_k\times \cT)_\ast \ar[d]\\
\ast\ar[r]&c_{k/\partial c_k}\smsh \cT_\ast~.}
\]
The vertical functors can be factored as two adjacent commutative squares
\[
\xymatrix{
(\partial c_k\times \cT)_\ast \ar[d]\ar[r]&(c_k\times \cT)_\ast \ar[d]\\
\partial c_{k,\ast} \ar[d]\ar[r]& c_{k,\ast} \ar[d]\\
\ast\ar[r]&c_{k/\partial c_k}\smsh \cT_\ast~.}
\]
Because the $\infty$-category of $(\infty,n+k)$-categories is Cartesian closed (\cite{rezk-n}), 
taking products preserves localizations.
Given the condition that the classifying space of $\cT$ is contractible (i.e., that the functor $\cT \ra \ast$ is a localization), the vertical functors $(\partial c_k\times \cT)_\ast \ra \partial c_{k,\ast}$ and $(c_k\times \cT)_\ast \ra c_{k,\ast}$ are localizations between pointed $(\oo,n+k)$-categories. 
A non-identity $i$-morphism in $(c_k\times\cT)_\ast$ is sent to an equivalence in $c_{k,\ast}$ if and only if it is in the image in $(\partial c_k \times\cT)_\ast$. 
It follows that the top commutative square is a pushout.
Consequently, the bottom commutative square is a pushout.
This is exactly the statement of the canonical functor $c_{k/\partial c_k}\smsh \cT_\ast \ra c_{k/\partial c_k}$ being an equivalence.

\end{proof}

\begin{proof}[Proof of Theorem~\ref{thm.cart.tang}]
Let $i\geq 0$.
The unique functor $c_i\to c_0$ determines a map between spaces: $\Fun_\ast\bigl( c_{k/\partial c_k}  ,  \cC \bigr)(c_0) \to \Fun_\ast\bigl( c_{k/\partial c_k}  ,  \cC \bigr)(c_i)$.
Unwinding the definition of the pointed $(\infty,n+k)$-category $\Fun_\ast\bigl( c_{k/\partial c_k}  ,  \cC \bigr)$, this map between spaces is identified as the map
{\Small
\begin{equation}\label{23}
\Fun_\ast\bigl( c_{k/\partial c_k}  ,  \cC \bigr) (c_0)
~\underset{(\ref{19})}\simeq~
\Map^{\ast/}\bigl( c_{k/\partial c_k} \smsh (c_0)_\ast ,  \cC \bigr) 
\longrightarrow
\Map^{\ast/}\bigl( c_{k/\partial c_k}\smsh (c_i)_\ast ,  \cC \bigr) 
~\underset{(\ref{19})}\simeq~
\Fun_\ast\bigl( (c_{k/\partial c_k})^{\adj}  ,  \cC \bigr) (c_i)~.
\end{equation}
}
Lemma~\ref{lem.smshcell} applied to $\cT=c_i$ gives that the above arrow is in fact an equivalence.
We conclude that each $i$-morphism in the $(\infty,n+k)$-category $\Fun_\ast\bigl( c_{k/\partial c_k}  ,  \cC \bigr)$ is invertible; this is to say that this pointed $(\infty,n+k)$-category is in fact a pointed space.  
Through the commutative diagram~(\ref{20}), and the Terminology~\ref{source.target} of $k$-endomorphisms and of objects in a higher category, there is a canonical identification between spaces
\begin{equation}\label{24}
\kEnd_\cC(\uno)
~\underset{\rm Term~\ref{source.target}}\simeq~
\Map^{\ast/}\bigl( c_{k/\partial c_k} ,  \cC \bigr) 
~\underset{(\ref{20})}\simeq~
\obj  \Fun_\ast\bigl( c_{k/\partial c_k}  ,  \cC \bigr) 
~\underset{\rm Term~\ref{source.target}}\simeq~
\Fun_\ast\bigl( c_{k/\partial c_k}  ,  \cC \bigr) (c_0)~.
\end{equation}
Concatenating~(\ref{23}) and~(\ref{24}) gives a canonical equivalence between pointed $(\infty,n+k)$-categories:
$
\kEnd_\cC(\uno)   \simeq  \Fun_\ast ( c_{k/\partial c_k}  ,  \cC )    .
$
Using that $\cC$ has adjoints, this identification is the further identification between pointed $(\infty,n+k)$-categories:
\begin{equation}\label{22}
\kEnd_\cC(\uno)~\simeq~\Fun_\ast\bigl( (c_{k/\partial c_k})^{\adj}  ,  \cC \bigr)~.
\end{equation}
The tangle hypothesis (Corollary~\ref{cor.univ.tang}) gives that the pointed functor $c_{k/\partial c_k} \to \widehat{\Tang_{n\subset n+k}^{\fr}}$ classifying the $k$-endomorphism $(\{0\}\subset \RR^k)$ determines an identification
\begin{equation}\label{21}
(c_{k/\partial c_k} )^{\adj}\xra{~\simeq~} \widehat{\Tang_{n\subset n+k}^{\fr}}
\end{equation}
between pointed $(\infty,n+k)$-categories with adjoints.
Inserting the identification~(\ref{21}) into~(\ref{22}) gives the sought identification.
\end{proof}

\subsection{The Cartesian enrichment of the cobordism hypothesis}
There is a natural enrichment of symmetric monoidal $(\oo,n)$-categories in $(\oo,n)$-categories:
\[
\xymatrix{
&\Cat_n \ar[dr]^-{\sf obj}\\
\Alg_{\sf Com}(\Cat_n)^{\op}\times \Alg_{\sf Com}(\Cat_n) 
\ar[ur]^-{\Fun^{\ot}}\ar[rr]_-{\Map^{\ot}}&& \spaces .
}
\]
This enrichment is defined by setting $\Fun^{\ot}(\fZ,\fX)$ to be the inverse limit of underlying $(\oo,n)$-categories:
\[
\Fun^{\ot}(\fZ,\fX)~:=~\underset{k\mapsto \oo}\limit \Fun_\ast\bigl(\fB^k \fZ, \fB^k\fX\bigr)~.
\]
Here, $\Fun_\ast\bigl(\fB^k \fZ, \fB^k\fX\bigr)$ is the underlying $(\oo,n)$-category of the natural $(\oo,n+k)$-category of pointed functors from $\fB^k \fZ$ to $\fB^k\fX$.
With respect to this enrichment, we have the following form of the cobordism hypothesis.

\begin{theorem}
Let $\fX$ be a symmetric monoidal $(\infty,n)$-category with adjoints and duals.
Assuming Conjecture \ref{conj.one}, evaluation on $\ast\in \Bord_n^{\fr}$ determines an equivalence
\[
\ev_\ast\colon \Fun^{\ot}\bigl(\Bord_n^{\fr},\fX\bigr)\xra{~\simeq~}\obj(\fX)
\]
between the $(\oo,n)$-category of $\fX$-valued symmetric monoidal functors and the space of objects of $\fX$. In particular, the $(\infty,n)$-category $\Fun^{\ot}(\Bord_n^{\fr},\fX)$ is an $\oo$-groupoid.
\end{theorem}
\begin{proof}
This follows from Theorem \ref{thm.cart.tang}.
\end{proof}

\subsection{The geometric enrichment of the tangle hypothesis}
We construct a second enrichment: this will constitute the dashed functor in a commutative diagram
\[
\xymatrix{
&&\Cat_{n}^{\ast/}\ar[dr]^-{\obj}\\
\Cat_{n+k}^{\adj, \ast/}\ar@{-->}[urr]^-{\Tang_{n\subset n+k}^{\fr}(-)}\ar[rrr]_-{\Map^{/ast/}(\Tang_{n\subset n+k}^{\fr},-)}&&&\spaces~.}
\]
That is, the functor will send $\cC$, a pointed $(\oo,n+k)$-category with adjoints, to a pointed $(\oo,n)$-category denoted
\[
\Tang_{n\subset n+k}^{\fr}(\cC)
\]
which has an identification of the underlying space of objects
\[
\obj\Bigl(\Tang_{n\subset n+k}^{\fr}(\cC)\Bigr) 
~\simeq~ 
\Map^{/ast/}\bigl(\Tang_{n\subset n+k}^{\fr}, \cC\bigr)~.
\]
We construct this from geometric products, as follows. 
Taking products implments a functor involving $\infty$-categories of solidly $(n+k)$-framed stratified manifolds
\[
\Mfd_k^{\sfr} \times \Mfd_n^{\sfr}
\xra{~\times~} 
\Mfd_{n+k}^{\sfr}~.
\]
On underlying stratified spaces, this functor is the usual product. 
On tangential structures, this uses the functor
\[
\Vect^{\sf inj}_{/\RR^k}\times \Vect^{\sf inj}_{/\RR^n}
\xra{~\oplus~}
\Vect^{\sf inj}_{/\RR^{n+k}}~.
\]
This associates the product solid $(n+k)$-framing $\sT_{X\times Y}\hookrightarrow \epsilon^{n+k}_{X\times Y}$ given a solid $k$-framing $\sT_X\hookrightarrow \epsilon^k_X$ and solid $n$-framing $\sT_Y \hookrightarrow \epsilon^n_Y$.

\begin{definition}[Assuming Conjecture \ref{conj.one}]\label{def.geom.enrich}
For $\uno\in \cC$ a pointed $(\oo,n+k)$-category with adjoints, the pointed $(\oo,n)$-category $\Tang_{n\subset n+k}^{\fr}(\cC)$ is the composite of the cellular realization functor $\btheta_{n,\emptyset}^{\op}\hookrightarrow \Mfd_n^{\vfr}\ra \Mfd_n^{\sfr}$ with
\[
\Tang_{n\subset n+k}^{\fr}(\cC)\colon \Mfd_n^{\sfr}\xra{~\RR^k\times -~}\Mfd_{n+k}^{\sfr}
\xra{~\int\cC~}
\spaces~.
\]
\end{definition}
\begin{remark}\label{rem.geom.enrich}
For $0\leq i\leq n$, the space of $i$-morphisms of $\Tang_{n\subset n+k}^{\fr}(\cC)$ is
\[
\Tang_{n\subset n+k}^{\fr}(\cC)(c_i) := \int_{\RR^k\times \DD^i}\cC~.
\]
In particular, for $i=0$ we identify the space of objects
\[
\obj\Bigl(\Tang_{n\subset n+k}^{\fr}(\cC)\Bigr)=\Tang_{n\subset n+k}^{\fr}(\cC)(c_0) 
~\simeq~ 
\kEnd_\cC(\uno)
\]
where the second equivalence is from Lemma~\ref{lemma.corep}.
\end{remark}

We obtain the following enriched version of the tangle hypothesis.

\begin{theorem}\label{thm.tang.geom.enrich}
Assuming Conjecture \ref{conj.one}, for each pointed $(\infty,n+k)$-category $\uno\in\cC$ with adjoints, evaluation at the $k$-endomorphism $(\{0\}\subset\RR^k) \in \Tang_{n\subset n+k}^{\fr}$ defines an equivalence between pointed $(\oo,n)$-categories
\[
\Tang_{n\subset n+k}^{\fr}(\cC) \ \simeq \  \Omega^k\cC
\]
between the geometric enrichment of Definition~\ref{def.geom.enrich} and the $k$-fold loops of $\cC$ at the distinguished object $\uno\in \cC$.
\end{theorem}
\begin{proof}
From \S3.4, the pointed $(\oo,n)$-category $\Omega^k\cC$ is the restriction to $\btheta_n^{\op}$ of the functor on vari-framed stratified $n$-manifolds
\[
\Mfd_n^{\vfr}\xra{\Omega^k\cC}\spaces
~,\qquad
D \mapsto \int_{\RR^k\times D} \cC~.
\]
In the case that $\cC$ is assumed to have adjoints, this restriction is canonically identified with the functor on $\btheta_n^{\op}$ defined as $\Tang_{n\subset n+k}^{\fr}$, through Remark~\ref{rem.geom.enrich}.
\end{proof}

\subsection{The geometric enrichment of the cobordism hypothesis}

Let $\Alg^{\sf dual}_{\sf Com}(\Cat_{n}^{\adj})$ denote the full $\oo$-subcategory of symmetric monoidal $(\oo,n)$-categories with adjoints for which every object has a dual. By passing to limits from the construction for tangle categories, we define a parallel geometric enrichment:
\[
\xymatrix{
&&\Cat_{n}^{\ast/}\ar[dr]^-{\obj}\\
\Alg^{\sf dual}_{\sf Com}(\Cat_{n}^{\adj})\ar@{-->}[urr]^-{\Bord_{n}^{\fr}(-)}\ar[rrr]_-{\Map^{\ot}(\Bord_{n}^{\fr},-)}&&&\spaces~.}
\]
This is as follows.
\begin{definition}[Assuming Conjecture \ref{conj.one}]\label{def.geom.enrich.bord}
For $\fX$ a symmetric monoidal $(\oo,n)$-category with adjoints and duals, the pointed $(\oo,n)$-category
\[
\Bord_n^{\fr}(\fX) ~:=~ \underset{k\to \infty}\limit \Tang_{n\subset n+k}^{\fr}\bigl(\fB^k\fX\bigr)
\]
is the sequential limit of the tangle categories $\Tang_{n\subset n+k}^{\fr}(\fB^k\fX)$, where $\fB^k \fX$ is the $k$-fold deloop of $\fX$.
\end{definition}

We have the following geometrically enriched version of the cobordism hypothesis.

\begin{theorem}
Let $\fX$ be a symmetric monoidal $(\infty,n)$-category with adjoints and duals.
Assuming Conjecture \ref{conj.one}, evaluation on $\ast\in \Bord_n^{\fr}$ determines an equivalence between $(\oo,n)$-categories
\[
\ev_\ast\colon \Bord_n^{\fr}(\fX)\xra{~\simeq~}\fX
\]
from the geometric enrichment of $\fX$-valued symmetric monoidal functors of Definition~\ref{def.geom.enrich.bord} and $\fX$ itself.
\end{theorem}
\begin{proof}
As in the proof of Theorem~\ref{B-bord}, we use the definition of $\Bord_n^{\fr}(\fX)$ as a sequential limit of tangle categories.
We then apply the form of the tangle hypothesis given by Theorem~\ref{thm.tang.geom.enrich}:
\begin{eqnarray}
\nonumber
\Bord_{n}^{\fr}(\fX)
&
\simeq
&
\underset{k\to \infty}\limit \Tang_{n\subset n+k}^{\fr}\bigl(\fB^k\fX\bigr)\\
\nonumber
&
\underset{\rm Thm~\ref{thm.tang.geom.enrich}}\simeq
&
\underset{k\to \infty}\limit \Omega^k\fB^k \fX
\\
\nonumber
&
\simeq
&
\underset{k\to \infty}\limit \fX\simeq \fX~.
\end{eqnarray}
\end{proof}


\begin{thebibliography}{99}

\bibitem[AW]{allendoerfer.weil} Allendoerfer, Carl; Weil, Andr\'e. The Gauss--Bonnet theorem for Riemannian polyhedra. Trans. Amer. Math. Soc. 53 (1943), 101--129.

\bibitem[At]{atiyah} Atiyah, Michael. Topological quantum field theories. Inst. Hautes \'Etudes Sci. Publ. Math. No. 68 (1988), 175--186 (1989).

\bibitem[AF1]{oldfact} Ayala, David; Francis, John. Factorization homology of topological manifolds.  J. Topol. 8 (2015), no. 4, 1045--1084.  

\bibitem[AF2]{II} Ayala, David; Francis, John. Factorization homology II: closed sheaves. In preparation.

\bibitem[AF3]{fibrations} Ayala, David; Francis, John. Fibrations of $\oo$-categories. Preprint, 2017.

\bibitem[AF4]{flagged} Ayala, David; Francis, John. Flagged higher categories. In preparation.

\bibitem[AFR1]{striat} Ayala, David; Francis, John; Rozenblyum, Nick. A stratified homotopy hypothesis. To appear, Journal of the European Mathematical Society. arXiv:1409.2857.

\bibitem[AFR2]{emb1a} Ayala, David; Francis, John; Rozenblyum, Nick. Factorization homology I: higher categories. Preprint, 2015.

\bibitem[AFT]{aft1} Ayala, David; Francis, John; Tanaka, Hiro Lee. Local structures on stratified spaces. Adv. Math. 307 (2017), 903--1028.

\bibitem[BaDo]{baezdolan} Baez, John; Dolan, James.
Higher-dimensional algebra and topological quantum field theory.
J. Math. Phys. 36 (1995), no. 11, 6073--6105. 

\bibitem[BS]{clark.chris} Barwick, Clark; Schommer-Pries, Christopher. On the unicity of the homotopy theory of higher categories. Preprint, 2013.

\bibitem[BeDr]{bd} Beilinson, Alexander; Drinfeld, Vladimir. Chiral algebras. American Mathematical Society Colloquium Publications, 51. American Mathematical Society, Providence, RI, 2004.

\bibitem[Be]{berger} Berger, Clemens. Iterated wreath product of the simplex category and iterated loop spaces. Adv. Math. 213 (2007), no. 1, 230--270.

\bibitem[BR1]{bergnerrezk1} Bergner, Julia; Rezk, Charles. Comparison of models for $(\oo,n)$-categories, I. Geom. Topol. 17 (2013), no. 4, 2163--2202.

\bibitem[BR2]{bergnerrezk2} Bergner, Julia; Rezk, Charles. Comparison of models for $(\oo,n)$-categories, II. Preprint.

\bibitem[BV]{bv} Boardman, J. Michael; Vogt, Rainer. Homotopy invariant algebraic structures on topological spaces. Lecture Notes in Mathematics, Vol. 347. Springer-Verlag, Berlin-New York, 1973. x+257 pp.

\bibitem[CF]{cranefrenkel} Crane, Louis; Frenkel, Igor. Four dimensional topological quantum field
theory, Hopf categories, and the canonical bases. Journal of Mathematical Physics, 35. 5136--5154, 1994.

\bibitem[CY]{craneyetter}  Crane, Louis; Yetter, David. A categorical construction of 4D topological quantum field theories. Quantum topology, 120--130, Ser. Knots Everything, 3, World Sci. Publ., River Edge, NJ, 1993.

\bibitem[Co]{costello} Costello, Kevin. Topological conformal field theory and Calabi--Yau categories.  Adv. Math. 210 (2007), no. 1, 165--214.

\bibitem[EM]{eliash} Eliashberg, Yakov; Mishachev, Nikolai. The space of framed functions is contractible. Essays in mathematics and its applications, 81--109, Springer, Heidelberg, 2012.

\bibitem[Fr1]{freed}  Freed, Daniel. Extended structures in topological quantum field theory. Quantum topology, 162--173, Ser. Knots Everything, 3, World Sci. Publ., River Edge, NJ, 1993.

\bibitem[Fr2]{freed2}  Freed, Daniel. Higher algebraic structures and quantization. Comm. Math. Phys. 159 (1994), no. 2, 343--398.

\bibitem[Fr3]{freed3}  Freed, Daniel. The cobordism hypothesis. Bull. Amer. Math. Soc. (N.S.) 50 (2013), no. 1, 57--92.

\bibitem[GMTW]{gmtw}  Galatius, S{\o}ren; Madsen, Ib; Tillmann, Ulrike; Weiss, Michael. The homotopy type of the cobordism category. Acta Math. 202 (2009), no. 2, 195--239. 


\bibitem[GH]{gepnerhaugseng} Gepner, David; Haugseng, Rune. Enriched $\oo$-categories via non-symmetric $\oo$-operads. Adv. Math. 279 (2015), 575--716. 

\bibitem[Jo1]{joyal1} Joyal, Andr\'e. Quasi-categories and Kan complexes. Special volume celebrating the 70th birthday of Professor Max Kelly. J. Pure Appl. Algebra 175 (2002), no. 1-3, 207--222.

\bibitem[Jo2]{joyaltheta} Joyal, Andr\'e. Disks, duality and Theta-categories. Unpublished work, 1997.

\bibitem[JT]{juer.tillmann} Juer, R.; Tillmann, U. Localisations of cobordism categories and invertible TFTs in dimension two. Homology Homotopy Appl. 15 (2013), no. 2, 195--225.

\bibitem[Ka]{kapustin}  Kapustin, Anton. Topological field theory, higher categories, and their applications. Proceedings of the International Congress of Mathematicians. Volume III, 2021--2043, Hindustan Book Agency, New Delhi, 2010. 

\bibitem[KV]{kv} Kapranov, Mikhail; Voevodsky, Vladimir. Braided monoidal 2-categories and Manin--Schechtman higher braid groups. J. Pure Appl. Algebra 92 (1994), no. 3, 241--267.

\bibitem[La]{lawrence} Lawrence, Ruth. Triangulations, categories and extended topological field theories. Quantum topology, 191--208, Ser. Knots Everything, 3, World Sci. Publ., River Edge, NJ, 1993.

\bibitem[Lu1]{HTT} Lurie, Jacob. Higher topos theory. Annals of Mathematics Studies, 170. Princeton University Press, Princeton, NJ, 2009. xviii+925 pp.

\bibitem[Lu2]{HA} Lurie, Jacob. Higher algebra. Preprint, 2016.

\bibitem[Lu3]{lurie.cobordism} Lurie, Jacob. On the classification of topological field theories. Current developments in mathematics, 2008, 129--280, Int. Press, Somerville, MA, 2009.

\bibitem[Mi]{milnor.top} Milnor, John. Topology from the differentiable viewpoint. Based on notes by David W. Weaver. Revised reprint of the 1965 original. Princeton Landmarks in Mathematics. Princeton University Press, Princeton, NJ, 1997.

\bibitem[Re1]{rezk}  Rezk, Charles. A model for the homotopy theory of homotopy theory. Trans. Amer. Math. Soc. 353 (2001), no. 3, 973--1007.

\bibitem[Re2]{rezk-n} Rezk, Charles. A Cartesian presentation of weak $n$-categories. Geom. Topol. 14 (2010), no. 1, 521--571.

\bibitem[RT1]{rt1} Reshetikhin, Nicolai; Turaev, Vladimir. Ribbon graphs and their invariants derived from quantum groups. Comm. Math. Phys. 127 (1990), no. 1, 1--26.

\bibitem[RT2]{rt2} Reshetikhin, Nicolai; Turaev, Vladimir. Invariants of 3-manifolds via link polynomials and quantum groups. Invent. Math. 103 (1991), no. 3, 547--597.

\bibitem[SP]{schommer} Schommer-Pries, Christopher. The classification of two-dimensional extended topological field theories. Thesis (Ph.D.)--University of California, Berkeley. 2009.

\bibitem[Se1]{segalconformal} Segal, Graeme. The definition of conformal field theory. Topology, geometry and quantum field theory, 421--577, London Math. Soc. Lecture Note Ser., 308, Cambridge Univ. Press, Cambridge, 2004.

\bibitem[Se2]{segallocal} Segal, Graeme. Locality of holomorphic bundles, and locality in quantum field theory. The many facets of geometry, 164--176, Oxford Univ. Press, Oxford, 2010. 

\bibitem[Si]{simpson} Simpson, Carlos. Homotopy theory of higher categories, New Mathematical Monographs, no. 19, Cambridge University Press, 2011.

\bibitem[Sm]{smale} Smale, Stephen. Generalized Poincar\'e's conjecture in dimensions greater than four. Ann. of Math. (2) 74 (1961), 391--406.

\bibitem[Ti]{tillmann} Tillmann, Ulrike. The classifying space of the $(1+1)$-dimensional cobordism category. J. Reine Angew. Math. 479 (1996), 67--75.

\bibitem[We]{weyl} Weyl, Hermann. On the volumes of tubes. Amer. J. Math. 61 (1939), 461--472.

\end{thebibliography}
\end{document}